\documentclass[reqno]{amsart}
\usepackage{a4wide}
\usepackage{pstricks}
\usepackage{graphics,graphicx}
\usepackage[english]{babel}
\usepackage{amsfonts,amsthm,amssymb,amsmath,mathrsfs}

\usepackage[colorlinks=false, linktocpage=true]{hyperref}
\newtheorem{thm}{Theorem}[section]
\newtheorem{cor}[thm]{Corollary}
\newtheorem{lem}[thm]{Lemma}

\theoremstyle{definition}

\theoremstyle{remark}
\newtheorem{rem}[thm]{Remark}
\numberwithin{equation}{section}

\newcommand{\Real}{\mathbb{R}}

\begin{document}
\title[Control and Stabilization for linear dispersive PDE's on $\mathbb{T}^2$]{Exact controllability and stabilization for linear dispersive PDE's on the two-dimensional torus}


%
\author{Francisco J. Vielma-Leal}\address{Departamento de Matemática, Facultad de Ciencias Naturales, Matemática y del Medio Ambiente, Universidad Tecnológica Metropolitana, Santiago, Chile.}
\email{fvielma@utem.cl}

\author{Ademir Pastor}\address{IMECC-UNICAMP, Rua S\'ergio Buarque de Holanda, 651, 13083-859, Campinas, SP, Brasil.}
\email{apastor@ime.unicamp.br}
%
%
\begin{abstract}
The moment method is used to prove the exact controllability of a wide class of bidimensional linear dispersive PDE's posed on the two-dimensional torus  $\mathbb{T}^{2}.$ The control function is considered to be acting on a small vertical and horizontal strip of the torus. Our results apply to several well-known models including some bidimesional extensions of the Benajamin-Ono and Korteweg-de Vries equations. As a by product, the exponential stabilizability with any given decay rate is also established in  $H^{s}_{p}(\mathbb{T}^{2}),$ with $s\geq 0,$ by constructing an appropriated feedback control law.
\end{abstract}
%
%
\maketitle

\section{Introduction}

The controllability and stabilizability for the linear Schrödinger equation on higher dimensions have been intensively studied during the last years, see for instance \cite{Fabre, Lasiecka and Triggiani, Lebeau, Phung,  Taufer} and references therein. When the problem is posed on a periodic domain, there are   pioneering works on this issue developed by the authors in \cite{Dehman Gerard and Lebeau, Camille, Miller} for the linear and nonlinear Schrodinger equations in dimensions 2 and 3 (see also \cite{Camille 1}). However, as far as we know,
there are a few works addressing the problems of exact controllability and asymptotic stabilization for bidimensional linear dispersive-type equations on a periodic setting. To the best of our knowledge, the only work dealing with this problem for a different dispersive model is the recent one in \cite{Rivas and Chenmin}, where the authors study the internal controllability of a non-localized solution for the linear and non-linear Kadomtsev-Petviashvili II equation.

As is well known, the first step to study the controllability of a nonlinear equation is to understand the controllability of the corresponding linear equation. So, our main goal in this paper is to investigate the control properties of a quite general class of linear dispersive equations on the  two-dimensional torus $\mathbb{T}^{2}:=\mathbb{R}^{2}/(2\pi \mathbb{Z})^{2}.$ More precisely, we are interested in the equation 
\begin{equation}\label{FEQ}
\partial_{t}u-\partial_{x}\mathcal{L}u=0, \;\;(x,y)\in \mathbb{T}^{2},\;t\in \mathbb{R},
\end{equation}
where $u\equiv u(x,y,t)$ denotes a real-valued function of three real variables $x,y$ and $t,$ and $\mathcal{L}$ denotes a linear Fourier multiplier operator. We assume that such multiplier $\mathcal{L}$ is of ``order'' $r-1,$ for some $r\in \mathbb{R},$ with   $r\geq 1.$ This means that the symbol $b:\mathbb{Z}^{2} \to \mathbb{R}$ satisfies
\begin{equation}\label{symbol}
	\widehat{\mathcal{L}u}(\mathbf{k})=b(\mathbf{k})\widehat{u}(\mathbf{k}),\;\;\forall\;\mathbf{k}=(k_{1},k_{2})\in \mathbb{Z}^{2},
\end{equation}
where $\widehat{u}$ stands for the (periodic) Fourier transform of $u$ (see \eqref{Ft}), and
\begin{equation}\label{Scont}
	|b(\mathbf{k})|\leq C |\mathbf{k}|^{r-1},
\end{equation}
for some positive constant $C$  and
$|\mathbf{k}|:=\sqrt{k_{1}^{2}+k_{2}^{2}}\geq N_{0},$ for some $N_{0}\geq 0.$ 
In view of the Parseval identity, it is easy to see that $\mathcal{L}$ is a self-adjoint operator on $L_{p}^{2}(\mathbb{T}^{2})$ (see Section \ref{preliminares} for notations)  and  commutes with derivatives.

There are several models that fit in the abstract form \eqref{FEQ}. For instance, the bidimensional versions of the Benjamin-Ono (BO) and Korteweg-de Vries (KdV) equations on a periodic setting. Specifically, the 2D Benjamin-Ono  ($\mathcal{L}=\mathcal{H}^{(x)}\partial_{y}$) and the Zakharov-Kuznetsov ($\mathcal{L}=-\Delta$) equations, where $\mathcal{H}^{(x)}$ denotes the Hilbert transform with respect to the $x$-variable and $\Delta$ is the bidimensional Laplacian operator. More general equations can also be written in the abstract form \eqref{FEQ}, for example, the Benjamin-Ono-Zakharov-Kuznetsov 
($\mathcal{L}=\mathcal{H}^{(x)}\partial_{x}-\partial_{y}^{2}$) and the dispersion generalized BOZK ($\mathcal{L}=D^{\alpha}_{x}-\partial_{y}^{2}$) equations, where $D^{\alpha}_{x}$ is defined  for $\alpha>0.$ As we will see below, all these equations may be treated in a single way.

As usual, the idea to study the controllability of \eqref{FEQ} is to add a forcing term $f\equiv f(x,y,t)$ as a control input. So, we shall consider the following non-homogeneous initial-value problem (IVP):
\begin{equation}\label{FEQ1}
	\partial_{t}u-\partial_{x}\mathcal{L}u=f,\;\;\; u(x,y,0)=u_{0}(x,y), \;\;\;(x,y)\in \mathbb{T}^{2},\;t\in \mathbb{R},
\end{equation}
for some suitable  control $f$. Here, we will assume that $f$ is acting on a small set composed by the union of a vertical and a horizontal strip; this means, $f$ is assumed to be supported on a set of the form $\left((\omega_{1}\times \mathbb{T}) \cup (\omega_{2}\times \mathbb{T})\right) \subset \mathbb{T}^{2},$ where $\omega_{1}$ and $\omega_{2}$ are small open intervals in $\mathbb{T}$ (see Figure \ref{regioncontrol}).
\begin{figure}[h!]
\begin{center}
	\includegraphics[width=9cm]{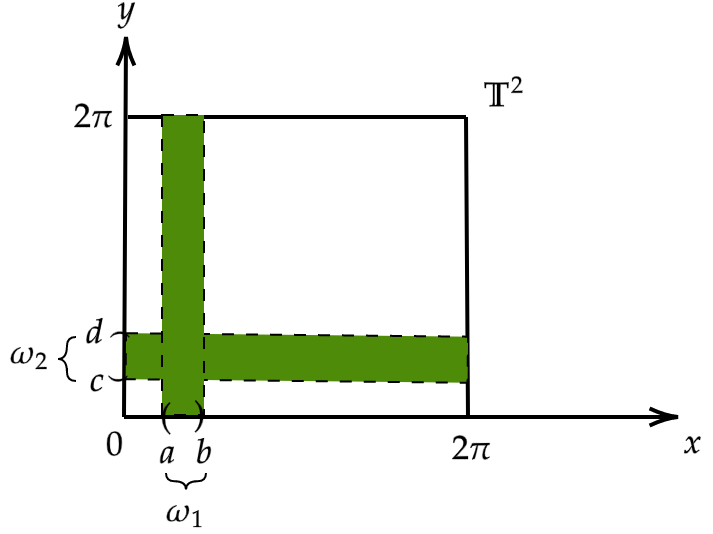}
	\caption{Region where the control $f$ is acting.}\label{regioncontrol}
\end{center}
\end{figure}

Note that \eqref{FEQ} conserves the total mass, that is, the quantity
$$
	\int_{\mathbb{T}^{2}} u(x,y,t)\;dxdy
$$
is conserved by any solution of \eqref{FEQ}. In order to keep the mass conserved in the control system \eqref{FEQ1},
we demand the function $f$ to satisfy
\begin{equation}\label{wcont}
	\int_{\mathbb{T}^{2}} f(x,y,t)\;dxdy=0, \;\;\forall t\in \mathbb{R}.
\end{equation}
In this regard, we consider the control $f$ of the form $Gh$,
where $h$ is a function defined in $\mathbb{T}^{2}\times[0,T]$ and the operator $G:H^{s}_{p}(\mathbb{T}^{2})\to H^{s}_{p}(\mathbb{T}^{2}),$ $s\geq 0,$ is defined in the following way: let $g_{1}$ and $g_{2}$ be  non-negative real-valued functions in $C^{\infty}(\mathbb{T})$ such that
\begin{equation}\label{med1}
	2\pi \widehat{g_{1}}(0)=\int_{\mathbb{T}}g_{1}(x)dx=1,
\end{equation}
\begin{equation}\label{med2}
	2\pi \widehat{g_{2}}(0)=\int_{\mathbb{T}}g_{2}(y)dy=1.
\end{equation}
Assume $\text{supp} \;g_{1}=\overline{\omega_{1}}\subset \mathbb{T}$ and $\text{supp} \;g_{2}=\overline{\omega_{2}}\subset \mathbb{T},$ where $\omega_{1}=\{x\in \mathbb{T}: g_{1}(x)>0\}$  and $\omega_{2}=\{y\in \mathbb{T}: g_{2}(y)>0\}$ are open intervals. Now, we define the operator $G$ as
\begin{equation}\label{opdef}
	G(\phi)(x,y):=g_{2}(y)\;G_{1}(\phi)(x,y)+g_{1}(x)\;G_{2}(\phi)(x,y), \;\;\phi \in  H^{s}_{p}(\mathbb{T}^{2}),\;s\geq 0,
\end{equation}
where
\begin{equation}\label{opG1}
	G_{1}(\phi)(x,y)=\frac{1}{2\pi} \phi(x,y)- \frac{1}{(2\pi)^{2}}\; \int_{0}^{2\pi} \phi(x',y)dx',  
\end{equation}
and 
\begin{equation}\label{opG2}
	G_{2}(\phi)(x,y)=\frac{1}{2\pi} \phi(x,y)- \frac{1}{(2\pi)^{2}}\; \int_{0}^{2\pi} \phi(x,y')dy'.
\end{equation}

It is easy to see that $G$ is linear and self-adjoint as an operator from $L^{2}_{p}(\mathbb{T}^{2})$ into $L^{2}_{p}(\mathbb{T}^{2}).$ In addition, it is bounded in $H_{p}^{s}(\mathbb{T}^{2})$, that is, there exists a constant $C$ depending only on $s,g_{1}$, and $g_{2}$ such that 
\begin{equation}\label{opl}
	\|G(\phi)\|_{H_{p}^{s}(\mathbb{T}^{2})}\leq C \|\phi\|_{H_{p}^{s}(\mathbb{T}^{2})}.
\end{equation}
Since we are setting $f=Gh$, the function $h$ can now be considered as the new control function and for each $t\in [0,T]$ we have that \eqref{wcont} holds. 

Next, we specify the problems that we address in this work, which are fundamental in control theory:

\vskip.2cm

\noindent \textbf{Exact controllability problem:} Let $s\geq 0$ and $T>0$ be given. Assume $u_{0}$ and $u_{1}$ belong to $H^{s}_{p}(\mathbb{T}^{2})$ with $\widehat{u_{0}}(0,0)=\widehat{u_{1}}(0,0).$ Can one find a control input $h$ such that the unique solution of the initial-valued problem (IVP)
\begin{equation}\label{pcont}
\begin{cases}
	\partial_{t}u-\partial_{x}\mathcal{L}u=G(h),\;\;(x,y)\in \mathbb{T}^{2},\;t\in \mathbb{R}, \\
	u(x,y,0)=u_{0}(x,y), 
\end{cases}
\end{equation}
is defined until time $T$ and satisfies $u(x,y,T)=u_{1}(x,y),$ for all $(x,y)\in \mathbb{T}^{2}$?

\vskip.2cm
\noindent \textbf{Asymptotic stabilizability problem:}
Let $s\geq 0$ and $u_{0}\in H^{s}_{p}(\mathbb{T}^{2})$ be given. Can one define a feedback control law $f=G(Ku)$ for some linear operator $K,$ such that the resulting closed-loop system
\begin{equation}\label{stabp} 
\begin{cases}
	\partial_{t}u-\partial_{x}\mathcal{L}u=G(Ku),\;\;(x,y)\in \mathbb{T}^{2},\;t\in \mathbb{R},& \\
	u(x,y,0)=u_{0}(x,y), 
\end{cases}
\end{equation}
is globally well-defined and asymptotically stable to an equilibrium point as $t \to +\infty$?

\vskip.2cm

Let us now describe our results. First, we state a result regarding controllability of equation \eqref{pcont}. Similar to the criteria in \cite{Vielma and Pastor}, the following results directly link the problem of controllability with some specific properties of the eigenvalues associated to the operator $\partial_{x}\mathcal{L}.$ To derive our first result, we assume that $\partial_{x}\mathcal{L}$ has a countable number of eigenvalues which, except on the coordinates axes, have finite  multiplicity. Specifically, we will assume the following hypothesis hold:
\vskip.2cm
\begin{itemize}
	\item [$(H1)$] $\partial_{x}\mathcal{L}\psi_{\mathbf{k}}=i\lambda_{\mathbf{k}} \psi_{\mathbf{k}},$ where  $\psi_{\mathbf{k}}$ is defined in \eqref{orto} and  $\lambda_{\mathbf{k}}=k_{1} b(\mathbf{k}),$ for all $\mathbf{k}=(k_{1},k_{2})\in \mathbb{Z}^{2}.$
\end{itemize}
\vskip.2cm
The eigenvalues in the sequence $\{i\lambda_{\mathbf{k}}\}_{\mathbf{k} \in \mathbb{Z}^{2}}$ are not necessarily distinct and we need to distinguish simple and multiple eigenvalues. Therefore, for each $\mathbf{k}' \in \mathbb{Z}^{2},$ we set       $I(\mathbf{k}') :=\{\mathbf{k} \in \mathbb{Z}^{2}: \lambda_{\mathbf{k}}=\lambda_{\mathbf{k}'} \}$ and $m(\mathbf{k}'):=\#I(\mathbf{k}'),$ where $\#I(\mathbf{k}')$ denotes the number of elements in $I(\mathbf{k}').$ In particular, $m(\mathbf{k}')=1$ if $i\lambda_{\mathbf{k}'}$ is a simple eigenvalue.
\vskip.2cm
\begin{itemize}
	\item [$(H2)$] For any $\mathbf{k}=  (k_{1},k_{2})\in \mathbb{Z}^{2},$ $\lambda_{\mathbf{k}}$ is even in the first variable and odd in the second one, it means,
	$\lambda_{(k_{1},k_{2})}=\lambda_{(-k_{1},k_{2})},$  and $\lambda_{(k_{1},-k_{2})}=-\lambda_{(k_{1},k_{2})}.$ Furthermore, for any $ \mathbf{k}= (k_{1},k_{2})\in \mathbb{Z}^{2},$ with $k_{1}\neq 0,$  we have that the unique entire solution $j_{2}\in \mathbb{Z}$ of equation 
	$\lambda_{(k_{1},j_{2})}=\lambda_{(k_{1},k_{2})}$ is $j_{2}=k_{2}.$ Also,
	 for any $\mathbf{k}=(k_{1},k_{2})\in \mathbb{Z}^{2},$ with  $k_{2}\neq 0,$ the unique entire solutions of equation 
	$\lambda_{(j_{1},k_{2})}=\lambda_{(k_{1},k_{2})}$ are $j_{1}=\pm k_{1}.$
\end{itemize}
\vskip.2cm
Assumptions $(H1)$ and $(H2)$ allows the eigenvalues $i\lambda_{\mathbf{k}}$ to have infinite multiplicity on the coordinate axes. Also, it may occur that $m(\mathbf{k}')\to \infty$ as $\lambda_{\mathbf{k}'}\to \infty$ (see Subsection \ref{BOexample} where a particular example involving the 2D-BO equation is given). Both properties make the problem of exact controllability  associated to \eqref{FEQ1} a great challenge and require us to find strong control functions $f$ acting not only on a small open  subset of $\mathbb{T}^{2}$ but on small subsets of vertical and horizontal strips of the 2-torus (see Figure \ref{regioncontrol}).

If we count only the distinct eigenvalues, we  obtain a countable  (maximal) set $\mathbb{J}\subset \mathbb{Z}^{2}$ and a sequence $\{\lambda_{\mathbf{k}}\}_{\mathbf{k}\in \mathbb{J}},$ with the property that $\lambda_{\mathbf{k}} \neq \lambda_{\mathbf{k}'}$ for any $\mathbf{k}, \mathbf{k}' \in \mathbb{J}$ with $\mathbf{k} \neq \mathbf{k}'.$ Now we are able to state our first result.

\begin{thm}\label{crtI}
	Let $s\geq 0$ and assume $(H1)$ and $(H2).$ Suppose that 
  \begin{equation}\label{gammacon}
	\begin{split}
		\gamma&:=\inf_{\substack{\mathbf{k},
				\mathbf{k}'\in \mathbb{J}\\ \mathbf{k}\neq \mathbf{k}'}}|\lambda_{\mathbf{k}}- \lambda_{\mathbf{k}'}|>0
	\end{split}
\end{equation}
and define
\begin{equation}\label{gammalinhacond}
	\begin{split}
		\gamma'&:=\underset{S\subset \mathbb{J}}{\sup}\;
		\underset{\mathbf{k} \neq \mathbf{k}'}{\underset{ \mathbf{k}, \mathbf{k}' \in \mathbb{J}\backslash S}{\inf}}
		|\lambda_{\mathbf{k}}- \lambda_{\mathbf{k}'}|,	
	\end{split}
\end{equation}
where $S$ runs over all finite subsets of $\mathbb{J}.$ Then, for any $T>\frac{2 \pi}{\gamma'}$	 and for each $u_{0},u_{1}\in H_{p}^{s}(\mathbb{T}^{2})$ with $\widehat{u_{0}}(0,0)=\widehat{u_{1}}(0,0),$ there exists a function $h\in L^{2}([0,T];H^{s}_{p}(\mathbb{T}^{2}))$ such that the unique solution $u$ of the non-homogeneous system 
\begin{equation}\label{pcont1}
\begin{cases}
	u\in C\left([0,T];H^{s}_{p}(\mathbb{T}^{2})\right), & \\
	\partial_{t}u=\partial_{x}\mathcal{L}u+G(h)(t) \in H^{s-r}_{p}(\mathbb{T}^{2}),\;\; t\in(0,T), & \\
	u(0)=u_{0}. & 
\end{cases}
\end{equation}
satisfies $u(T)=u_{1}.$ Furthermore,
\begin{equation}\label{esth}
	\|h\|_{L^{2}([0,T];H_{p}^{s}(\mathbb{T}^{2}))} \leq \nu \left(\|u_{0}\|_{H^{s}_{p}(\mathbb{T}^{2})}+
\|u_{1}\|_{H^{s}_{p}(\mathbb{T}^{2})}	\right)
\end{equation}
for some positive constant $\nu \equiv \nu(s,g_{1},T).$
\end{thm}

Using symmetry, we can replace  $(H2)$ by the following hypothesis and still have a similar result of exact controllability for system \eqref{pcont}:
\vskip.2cm
\begin{itemize}
	\item [$(H3)$] For any $ \mathbf{k}=(k_{1},k_{2})\in \mathbb{Z}^{2},$ $\lambda_{\mathbf{k}}$ is odd in the first variable and even in the second one, that is,
	 $\lambda_{(-k_{1},k_{2})}=-\lambda_{(k_{1},k_{2})}$ and $\lambda_{(k_{1},k_{2})}=\lambda_{(k_{1},-k_{2})}.$ Furthermore, For any $\mathbf{k}=(k_{1},k_{2})\in \mathbb{Z}^{2},$ with $k_{1}\neq 0$  we have the unique entire solutions $j_{2} \in \mathbb{Z}$ of equation 
	$\lambda_{(k_{1},j_{2})}=\lambda_{(k_{1},k_{2})}$ are $j_{2}=\pm k_{2}.$ Moreover,
	 for any $\mathbf{k}= (k_{1},k_{2})\in \mathbb{Z}^{2},$ with  $k_{2}\neq 0$ the unique entire solution $j_{1} \in \mathbb{Z}$ of equation 
	$\lambda_{(j_{1},k_{2})}=\lambda_{(k_{1},k_{2})}$ is $j_{1}= k_{1}.$
\end{itemize}
\vskip.2cm
In this case, we denote by $\mathbb{I}$ the (maximal) subset of $\mathbb{Z}^{2}$ such that
the sequence $\{\lambda_{\mathbf{k}}\}_{\mathbf{k}\in \mathbb{I}}$ have the property  $\lambda_{\mathbf{k}} \neq \lambda_{\mathbf{k}'}$ for any $\mathbf{k}, \mathbf{k}' \in \mathbb{I}$ with $\mathbf{k} \neq \mathbf{k}'.$ We now estate our second result.

\begin{thm}\label{crtI1}
	Let $s\geq 0$ and assume $(H1)$ and $(H3).$ Suppose that 
	\begin{equation}\label{gammacon1}
		\begin{split}
			\gamma&:=\inf_{\substack{\mathbf{k}, \mathbf{k}'\in \mathbb{I}\\ \mathbf{k}\neq \mathbf{k}'}}|\lambda_{\mathbf{k}}- \lambda_{\mathbf{k}'}|>0
		\end{split}
	\end{equation}
	and define
	\begin{equation}\label{gammalinhacond1}
		\begin{split}
			\gamma'&:=\underset{S\subset \mathbb{I}}{\sup}\;
			\underset{\mathbf{k} \neq \mathbf{k}'}{\underset{ \mathbf{k}, \mathbf{k}' \in \mathbb{I}\backslash S}{\inf}}
			|\lambda_{\mathbf{k}} -\lambda_{\mathbf{k}'}|,	
		\end{split}
	\end{equation}
	where $S$ runs over all finite subsets of $\mathbb{I}.$ Then, for any $T>\frac{2 \pi}{\gamma'}$	the same conclusions of Theorem \ref{crtI} hold.
\end{thm}

The idea to prove Theorems \ref{crtI} and \ref{crtI1} is to use the moment method (see, for instance, \cite{Russell, Vielma and Pastor}). Combined with a generalization of Ingham's theorem  (see \cite{7}), the construction of the function $h$ reduces in analyzing the solutions of an algebraic equation or system of equations.

\begin{rem}
	Note if $\lambda_{\mathbf{k}}\in\mathbb{Z}$, for all $\mathbf{k}\in\mathbb{J}$ (or $\mathbb{I}$), then we always have $\gamma,\gamma'\geq1$. This situation occurs, for instance, when $\mathcal{L}$ is a differential operator. See Section \ref{appsec} for some examples.
\end{rem}

Attention is now turned to our stabilization result.   Choosing an appropriate linear bounded operator $K$ one is able to show that  the resulting closed-loop system is exponentially stable with an arbitrary exponential decay rate. More precisely,

\begin{thm}\label{estabilization}
	Let  $g_{1},\;g_{2}$ be as in \eqref{med1}-\eqref{med2}  and let  $s\geq 0,$ and $\lambda>0$ be given. Under the assumptions of Theorem \ref{crtI}  or  Theorem \ref{crtI1}, there exists a bounded linear  operator $K_{\lambda}$  from $H_{p}^{s}(\mathbb{T}^{2})$ to $H_{p}^{s}(\mathbb{T}^{2})$ such that
	the unique solution $u$ of the closed-loop system
	\begin{equation}\label{estag}
		\begin{cases}
			u\in C([0,+\infty);H^{s}_{p}(\mathbb{T}^{2})), \hbox{}\\
			\partial_{t}u(t)= \partial_{x}\mathcal{L}u(t)+ GK_{\lambda}u(t)\in H_{p}^{s-r}(\mathbb{T}^{2}), \quad t> 0,\\
			u(0)=u_{0}\in H_{p}^{s}(\mathbb{T}^{2}), 
		\end{cases}
	\end{equation}		
	satisfies
	$$\|u(\cdot,t)-\widehat{u_{0}}(0,0)\|_{H_{p}^{s}(\mathbb{T}^{2})}\leq
	Me^{-\lambda t}\|u_{0}-\widehat{u_{0}}(0,0)\|_{H_{p}^{s}(\mathbb{T}^{2})},$$
	for all $t\geq0,$ and some positive constant $M=M(g_{1},g_{2},\lambda, s).$
\end{thm}

\begin{rem}
	Using a simple feedback control law $Ku=-G^{\ast},$ where $G^\ast$ denotes the adjoint operator of $G$, we can prove that the closed-loop system \eqref{stabp} is exponentially stable in $L^{2}_{p}(\mathbb{T}^{2})$ for some exponential decay rate by using similar arguments as in \cite[Theorem 5.4]{Manhendra and Francisco} (see also \cite{Linares Ortega, Russell and Zhang}).
\end{rem}

The paper is organized as follows. In Section \ref{preliminares} we introduce the basic notation and review some definitions related to periodic functions. In Section \ref{wellp} we just prove the well-posedness of our associated IVPs. Theorems \ref{crtI} and \ref{crtI1} is then proved in Section \ref{exactcontrol}. Section \ref{stabsec} is dedicated to prove the stabilization result. Finally, in Section \ref{appsec} we apply our results to prove the exact controllability and exponential stabilization for some well known models.

\section{Preliminaries}\label{preliminares}
In this section we introduce some basic notation and summarize some important results related with the theory of distributions on the two-dimensional torus (2-torus, for short). We will use $\mathbf{k}$ and $\mathbf{j}$ for generic points $(k_1,k_2)$ and $(j_1,j_2)$ in $\mathbb{Z}^2$. For  multi-indices  
 $\boldsymbol{\alpha}=(\alpha_{1},\alpha_{2})$ and
 $\boldsymbol{\beta}=(\beta_{1},\beta_{2})$ in $\mathbb{N}^{2}$
we say that $\boldsymbol{\beta} \leq \boldsymbol{ \alpha }$ if and only if $ \beta_{i} \leq \alpha_{i} ,\;i=1,2.$ Also, we define
$\boldsymbol{\alpha}!=\alpha_{1}!\alpha_{2}!$ and $|\boldsymbol{\alpha}|=\alpha_{1}+\alpha_{2}.$ Given two vectors $\mathbf{x}=(x_{1},y_{1})$ and $\mathbf{x}'= (x_{2},y_{2})$ in  $\mathbb{R}^{2},$
$\mathbf{x}\cdot \mathbf{x}'=x_{1}x_{2}+y_{1}y_{2},$
denotes the usual inner product. Also, $|\mathbf{x}|$ denotes the usual Euclidean norm of $\mathbf{x}.$

\subsection{Distributions on the 2-torus} Here we recall some aspects of the Fourier analysis on the torus as well as some properties of the periodic distributions. The details may be found in \cite[Chapter 3]{Loukas Grafakos}.
The space of test functions on the 2-torus is the space $C^{\infty}(\mathbb{T}^{2})$ of all $C^{\infty}$ functions that are $2\pi$-periodic in every coordinate. The topology generated by the family of semi-norms $\sup|D^\mathbf{\boldsymbol{\alpha}}f(\mathbf{x})|$ allows one to see $C^\infty(\mathbb{T}^2)$ as a locally convex topological space.
We denote by $\mathcal{G}(\mathbb{Z}^{2})$
the space of the \emph{rapidly decreasing sequences on $\mathbb{Z}^{2}$} of all complex-valued sequences $\mu=\left\{\mu_{ \mathbf{k} }\right\}_{\mathbf{k}\in \mathbb{Z}^{2}}$ such that 
\begin{equation}\label{rde}
	\sum_{\mathbf{k}\in \mathbb{Z}^{2}}|\mathbf{k}^{\boldsymbol{\alpha}} |\;|\mu_{\mathbf{k}}|<\infty,\;\;\text{for all multi-index}\; \boldsymbol{\alpha} \in \mathbb{N}^{2}.
\end{equation}
Recall that $\mathbf{k}^{\boldsymbol{\alpha}}=(k_1,k_2)^{(\alpha_{1},\alpha_{2})}=k_1^{\alpha_{1}}k_2^{\alpha_{2}}$.
The space $\mathcal{G}(\mathbb{Z}^{2})$ is a Hausdorff locally convex topological space with the topology induced by the family of semi-norms 
$$
\|\mu\|_{\infty,\boldsymbol{\alpha}}:=\sup\limits_{\mathbf{k}\in \mathbb{Z}^{2}}\left(|\mathbf{k}^{ \boldsymbol{\alpha} }|\;|\mu_{\mathbf{k}}|\right),
$$
where $ \boldsymbol{\alpha} $ ranges over all multi-indices in $\mathbb{N}^{2}.$  The dual spaces of $C^{\infty}(\mathbb{T}^{2})$ and $\mathcal{G}(\mathbb{Z}^{2})$ under these topologies are denoted by $\mathcal{D}'(\mathbb{T}^{2})$ (the space of all \emph{distributions on $\mathbb{T}^{2}$}) and 
$\mathcal{G}'(\mathbb{Z}^{2})$ (the space of \emph{sequences with slow growth on $\mathbb{Z}^{2}$}), respectively. $C^{\infty}(\mathbb{T}^{2})$ is dense in $\mathcal{D}'(\mathbb{T}^{2})$ and we can define the usual operations of differentiation, translation, reflection, convolution and multiplication.
The Fourier transform of $u \in \mathcal{D}'(\mathbb{T}^{2})$ is the sequence $\left\{ \widehat{u}(\mathbf{k})\right\}_{\mathbf{k}\in \mathbb{Z}^{2}}$ defined as
	\begin{equation}\label{Ft}
	\widehat{u}(\mathbf{k})=\frac{1}{(2\pi)^{2}} \left\langle u, e^{-i\mathbf{k}\cdot \mathbf{x}}\right\rangle,\;\;\mathbf{k}=(k_{1},k_{2})\in \mathbb{Z}^{2},\;\;\mathbf{x}=(x,y)\in \mathbb{T}^{2},
	\end{equation}
where $\left\langle \cdot, \cdot \right\rangle$ denotes the pairing between 
$\mathcal{D}'(\mathbb{T}^{2})$ and $C^{\infty}(\mathbb{T}^{2}).$ 
The map $^{\wedge}:\mathcal{D}'(\mathbb{T}^{2}) \to \mathcal{G}'(\mathbb{Z}^{2})$ is a linear bijection with inverse $^{\vee}: \mathcal{G}'(\mathbb{Z}^{2}) \to \mathcal{D}'(\mathbb{T}^{2})$ (the inverse Fourier transform) defined by
$$\eta = \left\{\eta_{\mathbf{k}}\right\}_{\mathbf{k}\in \mathbb{Z}^{2}} \mapsto \eta^{\vee}(\mathbf{x}):=\sum\limits_{\mathbf{k}\in \mathbb{Z}^{2}}\eta_{\mathbf{k}}\;e^{i\mathbf{k}\cdot\mathbf{x}},\;\;\forall\; \mathbf{x}=(x,y)\in \mathbb{T}^{2},$$
where the series converges in the sense of 
$ \mathcal{D}'(\mathbb{T}^{2})$. 

Recall that $L^{2}_{p}(\mathbb{T}^{2})$ (the standard space of square integrable $2\pi$-periodic functions) with complex inner product 
$$\left(u,v\right)_{L^{2}_p(\mathbb{T}^{2})}= \int\limits_{\mathbb{T}^{2}} u(x,y) \overline{v(x,y)}dxdy$$
is a Hilbert space.
If $u\in L^2_p(\mathbb{T}^2)$  then
	\begin{equation*}
	\widehat{u}(\mathbf{k})=\frac{1}{(2\pi)^{2}} \int_{\mathbb{T}^{2}} e^{-i\mathbf{k}\cdot \mathbf{x}}u(x,y)dxdy.
\end{equation*}

\subsection{Sobolev spaces and Fourier series}(See \cite[Chapter 3 $\&$ 4]{Michael Taylor}) In this subsection we will introduce the so-called Sobolev spaces of $L^{2}$-type on the 2-torus.
Given $s\in\Real$, the (periodic) \textit{Sobolev space of order $s$} is defined as 
$$H^{s}_{p}(\mathbb{T}^{2}):=\left \{
\displaystyle{u=\sum_{\mathbf{k}\in \mathbb{Z}^{2}} \widehat{u}(\mathbf{k}) \;e^{i\mathbf{k}\cdot\mathbf{x}}\in \mathcal{D}'(\mathbb{T}^{2}) } \Bigl| \|u\|_{H^{s}_{p}(\mathbb{T})}^{2}:=(2\pi)^{2}\sum_{\mathbf{k} \in \mathbb{Z}^{2}}
(1+|\mathbf{k}|)^{2s}|\widehat{u}(\mathbf{k})|^{2}<\infty \right\}.$$
The space $H^{s}_{p}(\mathbb{T}^{2}) $ is a  Hilbert space endowed
with the inner product
\begin{equation*}
	\displaystyle{(u\,, \,v)_{H^{s}_{p}(\mathbb{T}^{2})}= (2\pi)^{2} \sum_{\mathbf{k}\in \mathbb{Z}^{2}}
		(1+|\mathbf{k}|)^{2s}\widehat{u}(\mathbf{k})\;\overline{\widehat{v}(\mathbf{k})},\;\;v\in H^{s}_p(\mathbb{T}^{2}), \;u\in H^{s}_p(\mathbb{T}^{2}).}
\end{equation*}
For any $s\in \Real$, $(H^{s}_{p}(\mathbb{T}^{2}))'$, the topological dual of $H^{s}_{p}(\mathbb{T}^{2})$, is isometrically isomorphic to $H^{-s}_{p}(\mathbb{T}^{2})$, where the duality is implemented by the pairing
$$
\displaystyle{\langle u, v\rangle_{H^{-s}_{p}(\mathbb{T}^{2})\times H_{p}^{s}(\mathbb{T}^{2})}= (2\pi)^{2} \sum_{\mathbf{k}\in \mathbb{Z}^{2}}
	\widehat{u}(\mathbf{k})\;\overline{\widehat{v}(\mathbf{k})}},\;\;\text{for all}\; v \in H^{s}_{p}(\mathbb{T}^{2}),\;u\in H^{-s}_{p}(\mathbb{T}^{2}).
$$
If $s_{1}, \;s_{2}\in \mathbb{R}$ with $s_{1}\geq s_{2}$ then
$H_{p}^{s_{1}}(\mathbb{T}^{2}) \hookrightarrow H_{p}^{s_{2}}(\mathbb{T}^{2}),$ where the  embedding is dense. Also $H_{p}^{0}(\mathbb{T}^{2})$ is isometrically isomorphic to $L^{2}_{p}(\mathbb{T}^{2})$.

	It is well-known that any distribution $u\in H^{s}_{p}(\mathbb{T}^{2})$, $s\in\Real$, may be written as
	\begin{equation*}
		u=2\pi\sum_{\mathbf{k}\in\mathbb{Z}^{2}} \widehat{u}(\mathbf{k})\psi_{\mathbf{k}}= \lim\limits_{l \to \infty} S_{l}(u),
	\end{equation*}
	where the limit is taken in the sense of $\mathcal{D}'(\mathbb{T}^{2}),$ $S_{l}(u)$ 	is the $l$-th partial sum of the Fourier series associated to $u,$ defined by
		$$S_{l}(u)(\mathbf{x}) :=2\pi \sum\limits_{\substack{\mathbf{k}=(k_{1},k_{2}) \in \mathbb{Z}^{2}\\  |\mathbf{k}|\leq l}}  \widehat{u}(\mathbf{k}) \psi_{ \mathbf{k} },$$
and  $\left\{\psi_{\mathbf{k}}\right\}_{\mathbf{k}\in \mathbb{Z}^{2}}$ is a complete orthonormal  sequence in $L^{2}_{p}(\mathbb{T}^{2})$ (see \cite[Chapter 3 $\&$ 3.1]{Loukas Grafakos}) formed by  the complex-valued functions
	\begin{equation}\label{orto}
		\psi_{\mathbf{k}}(\mathbf{x})=\frac{1}{2\pi} e^{i\mathbf{k}\cdot\mathbf{x}},\;\;\mathbf{k}\in \mathbb{Z}^{2},\;\mathbf{x}=(x,y)\in \mathbb{T}^{2}.
	\end{equation}

\section{Well-posedness}\label{wellp}

Before establishing our main result concerning the exact controllability a we need a well-posedness theory associated with \eqref{FEQ}.
The following result states the well-posedness  with initial data in  $ H^{s}_{p}(\mathbb{T}^{2}),$ $s\in \mathbb{R}$.

\begin{thm}\label{wellpose}
Assume  $\mathcal{L}$ satisfies \eqref{symbol}-\eqref{Scont}.
Then for any and $u_{0}\in H_{p}^{s}(\mathbb{T}^{2})$, $s\in\Real$, the IVP
\begin{equation}\label{weelph2}
	\begin{cases}
		u\in C(\mathbb{R};H_{p}^{s}(\mathbb{T}^{2})), & \\
		\partial_{t}u=\partial_{x}\mathcal{L}u \in H^{s-r}_{p}(\mathbb{T}^{2}),\;\;\;t\in \mathbb{R},& \\
		u(0)=u_{0},
	\end{cases}
\end{equation}
has a unique solution.
\end{thm}
\begin{proof}
This follows from the standard semigroup theory. For the sake of completeness we bring some steps  (see for e.g. \cite{Cazenave and Haraux,  Iorio and Maga, Pazy} for more details). We consider $\partial_{x}\mathcal{L}:D(\partial_{x}\mathcal{L})\subseteq H^{s-r}(\mathbb{T}^{2}) \to H^{s-r}(\mathbb{T}^{2})$, where $D(\partial_{x}\mathcal{L})=H^{s}(\mathbb{T}^{2})$. Note that $D(\partial_{x}\mathcal{L})$ is dense in $H^{s-r}(\mathbb{T}^{2})$ and, using the definition of $\partial_{x}\mathcal{L}$ and the properties of the Fourier transform, we see that it is skew-adjoint, that is, for any $\varphi,\psi\in H^{s}(\mathbb{T}^{2})$,
\[
		\begin{split}
			\left(\partial_{x}\mathcal{L}\varphi, \psi\right)_{H^{s-r}(\mathbb{T}^{2})}=	-\left(\varphi, \partial_{x} \mathcal{L}\psi\right)_{H^{s-r}(\mathbb{T}^{2})}.
	\end{split}
\]
	Hence, by Stone's theorem it follows that $\partial_{x} \mathcal{L}$ generates a strongly continuous unitary group of contractions, say, $\left\{ U(t)\right\}_{t\in \mathbb{R}}$ on $H^{s-r}(\mathbb{T}^{2})$.  Theorem 3.2.3 in \cite{Cazenave and Haraux} now implies the desired result.
\end{proof}

\begin{rem}
Using the Fourier transform, we may deduce that the unique solution given in Theorem \ref{wellpose} satisfies
\begin{equation*}
	\widehat{u(t)}(\mathbf{k})=e^{ik_{1}b(\mathbf{k})t} \widehat{u_{0}}(\mathbf{k}), \;\;\mathbf{k}=(k_{1},k_{2})\in \mathbb{Z}^{2},
\end{equation*}
or by taking the inverse Fourier transform,
$$u(t)=\left(e^{ik_{1}b(\mathbf{k})t} \widehat{u_{0}}(\mathbf{k})\right)^{\vee},\;\;t\in \mathbb{R}.$$
It means that
\begin{equation*}
	u(\mathbf{x},t)=\sum_{\mathbf{k}\in \mathbb{Z}^{2}} e^{ik_{1}b(\mathbf{k})t} \widehat{u_{0}}(\mathbf{k}) e^{i\mathbf{k}\cdot\mathbf{x}}, \;\;\mathbf{x}=(x,y)\in \mathbb{T}^{2},
\end{equation*}
must be the unique solution of IVP \eqref{FEQ} with initial data $u_{0}\in H^{s}_{p}(\mathbb{T}^{2}),$ where the series converges in the sense of $\mathcal{D}'(\mathbb{T}^{2})$. In particular,  the unitary group $U(t)$ is given by
\begin{equation*}
	t\mapsto U(t)\varphi:= e^{\partial_{x}\mathcal{L}t}\varphi=\left(e^{ik_{1}b(\mathbf{k})t}\widehat{\varphi}( \mathbf{k} )\right)^{\vee},\;\;\mathbf{k}=(k_{1},k_{2}) \in \mathbb{Z}^{2},
\end{equation*}
in such way that the solution  of IVP \eqref{FEQ} with initial data $u_{0}\in H^{s}_{p}(\mathbb{T}^{2}),$ becomes $u(t)=U(t)u_{0},$ for any $t\in \mathbb{R}.$ Also, recall that according to  \cite[Corollary 3.2.6]{Cazenave and Haraux} the adjoint operator $U(t)^{\ast}$ of $U(t)$  is linear bounded and satisfies $U(t)^{*}=U(-t)$, for all $t\in \mathbb{R}.$
\end{rem}

Next, we deal with the well-posedness of the non-homogeneous linear problem \eqref{pcont}. The following lemma is needed.

\begin{lem}\label{opeG}
	Let $s\geq 0$ and $G$ be defined as in \eqref{opdef}. Given any $T>0$, the operator $$G:L^{2}([0,T];H_{p}^{s}(\mathbb{T}^{2})) \longrightarrow L^{2}([0,T];H_{p}^{s}(\mathbb{T}^{2}))$$ is linear and bounded.
\end{lem}
\begin{proof}
It is clear that $G$ is linear. In addition, for any $h\in L^{2}([0,T];H_{p}^{s}(\mathbb{T}^{2}))$,  it follows from  \eqref{opl} that
\[
	\begin{split}
		\|Gh\|^{2}_{L^{2}([0,T];H_{p}^{s}(\mathbb{T}^{2}))}&= 
		\int_{0}^{T}\|Gh(t)\|^{2}_{H_{p}^{s}(\mathbb{T}^{2})} dt\\
&\leq C^{2}
\int_{0}^{T}\|h(t)\|^{2}_{H_{p}^{s}(\mathbb{T}^{2})} dt\\
&=C^{2} \|h\|^{2}_{L^{2}([0,T];H_{p}^{s}(\mathbb{T}^{2}))},
	\end{split}
\]
which yields the desired.
\end{proof}

\begin{thm}
	Let $T>0,$ $s\geq 0,$ $u_{0}\in H_{p}^{s}(\mathbb{T}),$ and $h\in L^{2}([0,T];H_{p}^{s}(\mathbb{T}^{2})).$ Then there exists a unique (mild) solution $u\in C([0,T];H_{p}^{s}(\mathbb{T}^{2}))$ of the IVP \eqref{pcont}.
\end{thm}
\begin{proof}
As in the proof of Lemma	\ref{opeG} we infer that $Gh\in L^{1}([0,T];H_{p}^{s}(\mathbb{T}^{2})).$
 Corollary 2.2 and Definition 2.3 in \cite[Chapter 4]{Pazy} imply that 
 $$u(t)=U(t)u_{0}+\int_{0}^{t}U(t-t')Gh(t')dt'$$ is the unique (mild) solution of
 $$
\begin{cases}
	u\in C\left([0,T];H^{s}_{p}(\mathbb{T}^{2})\right), & \\
	\partial_{t}u=\partial_{x}\mathcal{L}u+Gh(t) \in H^{s-r}_{p}(\mathbb{T}^{2}),\;\; t\in(0,T), & \\
	u(0)=u_{0}\in H_{p}^{s}(\mathbb{T}^{2}),
\end{cases}
$$
which in turn provides the solution of \eqref{pcont}.
\end{proof}

\section{Exact controllability results}\label{exactcontrol}
This section is devoted to prove Theorems \ref{crtI} and \ref{crtI1}, as an application of the  classical moment method (see \cite{Russell}). Before starting with the results, note that by replacing $u_{1}$ by $u_{1}-U(T)u_{0}$ if necessary, we may assume without loss of generality that $u_{0}=0.$ Consequently, in view of our assumptions, we may assume $\widehat{u_{1}}(0,0)=\widehat{u_{0}}(0,0)=0.$

Let us start by writing  the terminal  estate $u_{1}\in H_{p}^{s}(\mathbb{T}^{2})$ as
\begin{equation}\label{repreu}
	u_{1}(\mathbf{x})=	\sum\limits_{ \mathbf{k}\in \mathbb{Z}^{2}}\widehat{u_{1}}(\mathbf{k}) e^{i\mathbf{k}\cdot\mathbf{x}}
	=2\pi\sum\limits_{\mathbf{k}\in \mathbb{Z}^{2}}  \widehat{u_{1}}(\mathbf{k}) \;\psi_{\mathbf{k}}(\mathbf{x}),
\end{equation}
where the series converges in the distributional sense and  $\psi_{\mathbf{k}}$ is defined as in \eqref{orto}.  Next result characterizes the exact controllability of the linear non-homogeneous system \eqref{pcont1}. The idea of the proof  is similar to that of    
\cite[Lemma 4.1]{Manhendra and Francisco}, passing to the frequency space when necessary; so we omit the details.

\begin{lem}\label{charact}
	Let $s\geq 0$ and $T>0$ be given. Assume $u_{1}\in H_{p}^{s}(\mathbb{T}^{2})$ with $\widehat{u_{1}}(0,0)=0.$ Then, there exists $h\in L^{2}([0,T];H_{p}^{s}(\mathbb{T}^{2}))$ such that the solution of the IVP  \eqref{pcont1}
with initial data $u_{0}=0$ satisfies $u(T)=u_{1}$ if and only if 
\begin{equation}\label{CEQ}
	\int_{0}^{T}\left\langle Gh(\cdot,\cdot,t),\varphi(\cdot,\cdot,t)\right\rangle_{H^{s}_{p}(\mathbb{T}^{2})\times(H^{s}_{p}(\mathbb{T}^{2}))'}dt
	=\left\langle u_{1}(\cdot,\cdot),\varphi_{0}(\cdot,\cdot)\right\rangle_{H^{s}_{p}(\mathbb{T}^{2})\times(H^{s}_{p}(\mathbb{T}^{2}))'},
\end{equation}
for any $\varphi_{0}\in (H^{s}_{p}(\mathbb{T}^{2}))'$, where $\varphi$
is the solution of the adjoint problem
\begin{equation}\label{adsis}
	\begin{cases}
		\varphi\in C\left([0,T]:\left(H_{p}^{s}(\mathbb{T}^{2})\right)'\right),\\
		\partial_{t}\varphi=\partial_{x}\mathcal{L}\varphi \in H_{p}^{-s-r}(\mathbb{T}^{2}), \quad t>0,\\
		\varphi(T)=\varphi_{0}. 
	\end{cases}
\end{equation}
\end{lem}

The following characterization to show the existence of control for the linear system \eqref{pcont1} (with initial data $u_{0}=0$) is  a direct consequence of Lemma \ref{charact}. It provides a method to find the control function $h$ explicitly. For a proof in a very similar situation, we refer the reader to \cite[Lemma 4.3]{Vielma and Pastor}.

\begin{cor}[Moment Equation]\label{Momenteq}
	Let $s\geq 0$ and $T>0$ be given. If $u_1$ is written as in \eqref{repreu} and satisfies $\widehat{u_{1}}(0,0)=0,$ then the solution of \eqref{pcont1}
	with initial data $u_{0}=0$ satisfies $u(T)=u_{1}$ if and only if there exists $h\in L^{2}([0,T];H_{p}^{s}(\mathbb{T}^{2})) $ such that
	\begin{equation}\label{Momenteq1}
	\int_{0}^{T}\left(Gh(\cdot,\cdot,t), e^{-i\lambda_{\mathbf{k}}(T-t)} \psi_{\mathbf{k}}(\cdot,\cdot) \right)_{L^{2}(\mathbb{T}^{2})}	dt=2\pi \widehat{u_{1}}(\mathbf{k}),\;\;\forall\; \mathbf{k}=(k_{1},k_{2})\in \mathbb{Z}^{2},
	\end{equation}
where $\lambda_{\mathbf{k}}=k_{1} b(\mathbf{k}).$
\end{cor}

With Corollary \ref{Momenteq} in hand, we see that in order to prove Theorem \ref{crtI} (for instance) we only need to construct a control function $h$ satisfying relation \eqref{Momenteq1}. Before that, we need two additional results. The first one, gives some properties of how the operators $G_{1}$ and $G_{2}$  behave at complex exponential functions.

\begin{lem}\label{invertmatrix}
	For $G_{1}$ and $G_{2}$ as in \eqref{opG1}-\eqref{opG2}, define
	\begin{equation}\label{invertmatrix1}
		m^{j,k}_{1}:=\widehat{G_{1}(e^{ijx})}(k) =\frac{1}{2 \pi}\int_{0}^{2\pi}G_{1}(e^{ijx}) e^{-ikx}\;dx,\;\;\;\;j,k\in\mathbb{Z},
	\end{equation}
	\begin{equation}\label{invertmatrix2}
	m^{j,k}_{2}:=\widehat{G_{2}(e^{ijy})}(k)=\frac{1}{2 \pi}\int_{0}^{2\pi}G_{2}(e^{ijy}) e^{-iky}\;dy,\;\;\;\;j,k\in\mathbb{Z}.
\end{equation}
	Then
	\begin{itemize}
		\item [(i)]  $m_{n}^{j,0}=0,$  for all $j\in\mathbb{Z}$ and $n=1,2;$
		\item[(ii)] $m_{n}^{0,k}=0$, for all $k\in\mathbb{Z}$ and $n=1,2;$
		
		\item [(iii)] $m_{n}^{0,0}=0$, for $n=1,2;$
		
		\item [(iv)] If $j,k\in \mathbb{Z}$ with $j\neq 0$ and $k \neq 0,$ then
		\begin{equation}\label{posit1}
m_{n}^{j,k}=\begin{cases}
	\displaystyle{\frac{1}{2\pi}}, & \text{if}\; j=k,\\
	0, & \text{if}\; j\neq k, 
       \end{cases}
		\end{equation}
	for $n=1,2.$
	\end{itemize}
\end{lem}
\begin{proof}
The proof follows by direct calculations.
\end{proof}

The second result, gives the existence of a \textit{biorthogonal} basis with respect to  $\{e^{-i\lambda_{\mathbf{k}}t}\}_{\mathbf{k}\in \mathbb{J}}$.

\begin{lem}\label{basislemma}
Let $H:=\overline{\text{span}\{e^{-i\lambda_{\mathbf{k}} t}:\mathbf{k}\in \mathbb{J}\}}$ in $L^{2}([0,T]).$ There exists a unique basis $\{q_{\mathbf{k}}\}_{  \mathbf{k} \in \mathbb{J}}\subset H$ such that
\begin{equation}\label{dualbg}
	(e^{-i \lambda_{\mathbf{k}} t}\;,\;q_{ \mathbf{k}' })_{H}=\int_{0}^{T}e^{-i \lambda_{\mathbf{k}} t}\overline{q_{ \mathbf{k}' }}(t)\;dt=\delta_{ \mathbf{k} \mathbf{k}' },\;\;\;\mathbf{k} , \mathbf{k}' \in \mathbb{J},
\end{equation}
where $\delta_{ \mathbf{k} \mathbf{k}' }$ represents the Kronecker delta.
\end{lem}
\begin{proof}
The proof is quite well-known by now (see, for instance, \cite[Theorem 1.3]{Linares Ortega} or \cite[Thorem 1.3]{Vielma and Pastor}). The main idea is to use Ingham's Theorem (see {\cite[Theorem 4.6, pag. 67]{7}}) to show that $\{e^{-i\lambda_{\mathbf{k}}t}\}_{\mathbf{k}\in \mathbb{J}}$ is a Riesz basis of $H$. Then, the existence of  the biorthogonal basis follows from the standard theory in Hilbert spaces (see \cite{9}).
\end{proof}

We are now able to show the main results of this section regarding exact controllability.

\begin{proof}[Proof of Theorem \ref{crtI}]
	We prove this Theorem in  three steps. Recall, we are assuming $\widehat{u_{1}}(0,0)=0$ in \eqref{repreu}.\\

\noindent
{\bf{Step 1.}}  Construction of $h.$

Let $\{q_{ \mathbf{k} }\}_{ \mathbf{k} \in \mathbb{J}}$ be the sequence obtained in Lemma \ref{basislemma}. As a first step we will extend the definition of $q_{\mathbf{k}}$ for all $\mathbf{k} \in \mathbb{Z}^{2}$. We do this following the rule: given $\mathbf{k} \in \mathbb{Z}^{2}$ we know that there exists $ \mathbf{k}' \in \mathbb{J}$ such that $\lambda_{\mathbf{k}}= \lambda_{\mathbf{k}'},$ so we define 
\begin{equation}\label{defq}
	q_{\mathbf{k}}(t):=q_{ \mathbf{k}' }(t), \quad t\in [0,T].
\end{equation}

The control function $h$ is now defined as
\begin{equation}\label{thecontrol}
	h (x,y,t)=\sum_{\textbf{j} \in \mathbb{Z}^{2}} h_{\textbf{j}}\;\overline{q_{\textbf{j}}(t)}\;\psi_{\textbf{j}}(x,y),\;\;(x,y)\in \mathbb{T}^{2},
\end{equation}
for suitable  coefficients $h_{\textbf{j}}$'s to be determined later by using the Moment equation. Therefore, we note that for any $\mathbf{k}=(k_{1},k_{2})\in \mathbb{Z}^{2}$ the left-hand side of   \eqref{Momenteq1} can be rewritten as
\[
		\begin{split}
			I:=&
			\int\limits_{0}^{T}
			\left(Gh(x,y,t), e^{-i\lambda_{\mathbf{k}}(T-t)} \psi_{\mathbf{k}}(x,y)\right)_{L^{2}(\mathbb{T}^{2})}dt\\
			&=
			\int\limits_{0}^{T}\left(\sum_{\textbf{j}\in \mathbb{Z}^{2}} h_{\textbf{j}} \overline{q_{\textbf{j}}(t)} G(\psi_{\textbf{j}})(x,y,t),
			e^{-i\lambda_{\mathbf{k}}(T-t)} \psi_{\mathbf{k}}(x,y) \right)_{L^{2}_{p}(\mathbb{T}^{2})}dt\\
			&=
			\sum_{\textbf{j}\in \mathbb{Z}^{2}}h_{\textbf{j}} \int\limits_{0}^{T} \overline{q_{\textbf{j}}(t)} \;
			e^{i\lambda_{\mathbf{k}}(T-t)} \;dt \left(G(\psi_{\textbf{j}})(x,y),\; \psi_{\mathbf{k}}(x,y)\right)_{L^{2}(\mathbb{T}^{2})}\\
			&=
\sum_{\textbf{j}\in \mathbb{Z}^{2}}h_{\textbf{j}}\; 
e^{i\lambda_{\mathbf{k}}T} 
\left( \int\limits_{0}^{T}
e^{-i\lambda_{\mathbf{k}}t}\;
 \overline{q_{\textbf{j}}(t)} 
 \;dt \right)
  \left(G(\psi_{\textbf{j}})(x,y),\; \psi_{\mathbf{k}}(x,y)\right)_{L^{2}(\mathbb{T}^{2})},
	\end{split}
\]
where
\[
	\begin{split} 
		\left(G(\psi_{\textbf{j}})(x,y),\; \psi_{\mathbf{k}}(x,y)\right)_{L^{2}(\mathbb{T}^{2})}
		&= \left(g_{2}(y) \frac{e^{ij_{2}y}}{2 \pi}  G_{1}(e^{ij_{1}x}), \; \frac{e^{ik_{1}x} e^{ik_{2}y}}{2 \pi} 
		\right)_{L^{2}(\mathbb{T}^{2})} \\
		& \quad + 
		\left(g_{1}(x) \frac{e^{ij_{1}x}}{2 \pi}  G_{2}(e^{ij_{2}y}), \; \frac{e^{ik_{1}x} e^{ik_{2}y}}{2 \pi} 
		\right)_{L^{2}(\mathbb{T}^{2})} \\
		&=
		\widehat{g_{2}(y)}(k_{2}-j_{2}) \left( \frac{1}{2 \pi} \int_{0}^{2 \pi}
		 G_{1}(e^{ij_{1}x}) e^{-ik_{1}x} dx\right)\\
		 &\quad + 	\widehat{g_{1}(x)}(k_{1}-j_{1}) \left( \frac{1}{2 \pi} \int_{0}^{2 \pi}
		 G_{2}(e^{ij_{2}y}) e^{-ik_{2}y} dy\right)\\
		 &=\widehat{g_{2}(y)}(k_{2}-j_{2}) \; m_{1}^{j_{1},k_{1}}+  \widehat{g_{1}(x)}(k_{1}-j_{1})\; m_{2}^{j_{2},k_{2}}.
	\end{split}
\]
Hence,
\begin{equation}\label{thecontrol1g}
\begin{split}
	I&=
\sum_{\textbf{j}\in \mathbb{Z}^{2}}h_{\textbf{j}}\; 
e^{i\lambda_{\mathbf{k}}T}
 \left( \int\limits_{0}^{T}
e^{-i\lambda_{\mathbf{k}} t}\;
\overline{q_{\textbf{j}}(t)} 
\;dt \right)
\widehat{g_{2}(y)}(k_{2}-j_{2})\;
m_{1}^{j_{1},k_{1}}\\
	&\quad + 
\sum_{\textbf{j}\in \mathbb{Z}^{2}}h_{\textbf{j}}\; 
e^{i\lambda_{\mathbf{k}}T} 
\left( \int\limits_{0}^{T}
e^{-i\lambda_{\mathbf{k}} t}\;
\overline{q_{\textbf{j}}(t)} 
\;dt \right) 
\widehat{g_{1}(x)}(k_{1}-j_{1})\;
m_{2}^{j_{2},k_{2}},
		\end{split}
\end{equation}
with $m_{n}^{j_{n},k_{n}}$ defined as in \eqref{invertmatrix1}-\eqref{invertmatrix2} for $n=1,2.$\\

\noindent
{\bf{Step 2.}}
Construction of the coefficients  $h_{\textbf{j}}$. 

First of all, note that in order to prove the first part of the theorem,  identity \eqref{thecontrol1g}, Lemma \ref{invertmatrix} (ii), and Corollary \ref{Momenteq} yield that it suffices to choose $h_{\textbf{j}}$'s such that
\begin{equation}\label{h formg}
\begin{split}
2\pi \widehat{u_{1}}(\mathbf{k})
&=\sum_{\textbf{j}\in \mathbb{Z}^{\ast}\times \mathbb{Z}}h_{\textbf{j}}\; 
e^{i\lambda_{\mathbf{k}}T}
\left( \int\limits_{0}^{T}
e^{-i\lambda_{\mathbf{k}} t}\;
\overline{q_{\textbf{j}}(t)} 
\;dt \right)
\widehat{g_{2}(y)}(k_{2}-j_{2})\;
m_{1}^{j_{1},k_{1}}\\
&\quad + 
\sum_{\textbf{j}\in \mathbb{Z}\times \mathbb{Z}^{\ast}}h_{\textbf{j}}\; 
e^{i\lambda_{\mathbf{k}}T} 
\left( \int\limits_{0}^{T}
e^{-i\lambda_{\mathbf{k}} t}\;
\overline{q_{\textbf{j}}(t)} 
\;dt \right) 
\widehat{g_{1}(x)}(k_{1}-j_{1})\;
m_{2}^{j_{2},k_{2}},
\end{split}
\end{equation}
 for all $\mathbf{k}=(k_{1},k_{2})\in \mathbb{Z}^{2}.$ Recall that $\mathbb{Z}^*=\mathbb{Z}\setminus\{0\}$.

We will now show  that we may indeed choose $h_{\textbf{j}}$'s satisfying \eqref{h formg}. To see this, first observe that, since $\widehat{u_{1}}(0,0)=0$, part (i)  in Lemma \ref{invertmatrix} implies that \eqref{h formg} holds for $(k_{1},k_{2})=(0,0)$ independently of $h_{\textbf{j}}$'s. In particular, we may choose $h_{(0,0)}=0$. Next, from Lemma \ref{invertmatrix} (i)-(iv), if $\mathbf{k}=(k_{1},k_{2})\in \mathbb{Z}^{2}$ with $k_{1}\neq 0$ and $k_{2}=0,$  we see that \eqref{h formg} reduces to
\begin{equation}\label{h formg1}
	\begin{split}
		2\pi \widehat{u_{1}}(k_{1},0)
		&=\sum_{\textbf{j}\in \mathbb{Z}^{\ast}\times \mathbb{Z}}h_{\textbf{j}}\; 
		e^{i\lambda_{(k_{1},0)}T}
		\left( \int\limits_{0}^{T}
		e^{-i\lambda_{(k_{1},0)} t}\;
		\overline{q_{\textbf{j}}(t)} 
		\;dt \right)
		\widehat{g_{2}(y)}(-j_{2})\;
		m_{1}^{j_{1},k_{1}}\\
		&=\sum_{j_{2}\in \mathbb{Z}} h_{(k_{1},j_{2})}\; 
		e^{i\lambda_{(k_{1},0)}T}
		\left( \int\limits_{0}^{T}
		e^{-i\lambda_{(k_{1},0)} t}\;
		\overline{q_{(k_{1},j_{2})}(t)} 
		\;dt \right)
		\widehat{g_{2}(y)}(-j_{2})\;
		m_{1}^{k_{1},k_{1}}.
	\end{split}
\end{equation}
According to \eqref{dualbg}-\eqref{defq}, all terms of the series on the right-hand side of \eqref{h formg1} is zero, except for those  $j_{2}\in \mathbb{Z}$ such that 
$$
\lambda_{(k_{1},j_{2})}=\lambda_{(k_{1},0)}.
$$
In view of hypothesis $(H2),$ this holds only for $j_{2}=0.$ Hence, 
\begin{equation}\label{h formg2}
\begin{split}
	2\pi \widehat{u_{1}}(k_{1},0)
	&= h_{(k_{1},0)}\; 
	e^{i\lambda_{(k_{1},0)}T}
	\widehat{g_{2}(y)}(0)\;
	\frac{1}{2\pi}
\end{split}	
\end{equation} 
and, in view of \eqref{med2},
\begin{equation}\label{h formg3}
	\begin{split}
		 h_{(k_{1},0)}&=(2\pi)^{3} \widehat{u_{1}}(k_{1},0) 
		e^{-i\lambda_{(k_{1},0)}T}.
	\end{split}	
\end{equation}

Similarly, from Lemma \eqref{invertmatrix} (i)-(iv), if $\mathbf{k}=(k_{1},k_{2})\in \mathbb{Z}^{2}$ with $k_{1}= 0$ and $k_{2}\neq 0,$  we see that \eqref{h formg} reduces to
\begin{equation}\label{h formg4}
	\begin{split}
		2\pi \widehat{u_{1}}(0,k_{2})
		&=\sum_{j_{1}\in \mathbb{Z}} h_{(j_{1},k_{2})}\; 
		e^{i\lambda_{(0,k_{2})}T}
		\left( \int\limits_{0}^{T}
		e^{-i\lambda_{(0,k_{2})} t}\;
		\overline{q_{(j_{1},k_{2})}(t)} 
		\;dt \right)
		\widehat{g_{1}(x)}(-j_{1})\;
		m_{2}^{k_{2},k_{2}}.
	\end{split}
\end{equation}
In view of hypothesis $(H2),$ we see that $\lambda_{(j_{1},k_{2})}=\lambda_{(0,k_{2})}$ only for $j_{1}=\pm 0.$
Hence, from \eqref{dualbg}-\eqref{defq}, and \eqref{med1} we deduce that
\begin{equation}\label{h formg5}
	\begin{split}
		2\pi \widehat{u_{1}}(0,k_{2})
		&= h_{(0,k_{2})}\; 
		e^{i\lambda_{(0,k_{2})}T}
		\widehat{g_{1}(x)}(0)\;
		\frac{1}{2\pi}
	\end{split}	
\end{equation} 
or, equivalently,
\begin{equation}\label{h formg6}
	\begin{split}
		h_{(0,k_{2})}&=(2\pi)^{3} \widehat{u_{1}}(0,k_{2}) 
		e^{-i\lambda_{(0,k_{2})}T}.
	\end{split}	
\end{equation}
Finally, if $\mathbf{k}=(k_{1},k_{2})\in \mathbb{Z}^{2}$ with $k_{1}\neq 0$ and $k_{2}\neq 0,$ we get from \eqref{h formg} and Lemma \eqref{invertmatrix} (iv), that
\begin{equation}\label{h formg7}
	\begin{split}
		2\pi \widehat{u_{1}}(k_{1},k_{2})
		&=\sum_{j_{2}\in \mathbb{Z}} 
		h_{(k_{1},j_{2})}\; 
		e^{i\lambda_{(k_{1},k_{2})}T}
		\left( \int\limits_{0}^{T}
		e^{-i\lambda_{(k_{1},k_{2})} t}\;
		\overline{q_{(k_{1},j_{2})}(t)} 
		\;dt \right)
		\widehat{g_{2}(y)}(k_{2}-j_{2})\;
		m_{1}^{k_{1},k_{1}}\\
		&\quad + 
		\sum_{j_{1}\in \mathbb{Z}} 
		h_{(j_{1},k_{2})}\; 
		e^{i\lambda_{(k_{1},k_{2})}T} 
		\left( \int\limits_{0}^{T}
		e^{-i\lambda_{(k_{1},k_{2})} t}\;
		\overline{q_{(j_{1},k_{2})}(t)} 
		\;dt \right) 
		\widehat{g_{1}(x)}(k_{1}-j_{1})\;
		m_{2}^{k_{2},k_{2}}.
	\end{split}
\end{equation}
According to \eqref{dualbg}-\eqref{defq}, all terms in the sums are zero, except the ones where the entire variables $j_{1}, j_{2}$ solve the following equations: $$\lambda_{(k_{1},j_{2})}= \lambda_{(k_{1},k_{2})},$$
$$\lambda_{(j_{1},k_{2})}= \lambda_{(k_{1},k_{2})}.$$
In view of hypothesis $(H2),$ we have the solutions $j_{2}=k_{2}$ and $j_{1}=\pm k_{1},$ respectively. Hence, 
\begin{equation}\label{h formg8}
	\begin{split}
		2\pi \widehat{u_{1}}(k_{1},k_{2})
		&=
		h_{(k_{1},k_{2})}\; 
		e^{i\lambda_{(k_{1},k_{2})}T}
		\widehat{g_{2}(y)}(0)\;
	\frac{1}{2 \pi}
	 + 
		h_{(k_{1},k_{2})}\; 
		e^{i\lambda_{(k_{1},k_{2})}T}  
		\widehat{g_{1}(x)}(0)\;
	\frac{1}{2\pi}\\
		&\quad+
		h_{(-k_{1},k_{2})}\; 
		e^{i\lambda_{(k_{1},k_{2})}T}  
		\widehat{g_{1}(x)}(2k_{1})\;
		\frac{1}{2\pi},
	\end{split}
\end{equation}
or, equivalently,
\begin{equation}\label{h formg9}
	\begin{split}
		(2\pi)^{2} e^{-i\lambda_{(k_{1},k_{2})}T} \widehat{u_{1}}(k_{1},k_{2})
		&=
		h_{(k_{1},k_{2})}\; 
		\left(
		\widehat{g_{2}(y)}(0)
		+   
		\widehat{g_{1}(x)}(0)\right)
		+
		h_{(-k_{1},k_{2})}\;  
		\widehat{g_{1}(x)}(2k_{1}).
	\end{split}
\end{equation}
Note that in the left-hand side of \eqref{h formg9} appear the coefficients $h_{(k_{1},k_{2})}$ and $h_{(-k_{1},k_{2})}$; so, in order to determine them we will couple \eqref{h formg9} with another equation.  To do so, observe that 
\begin{equation*}
	\begin{split}
		2\pi \widehat{u_{1}}(-k_{1},k_{2})
		&=\sum_{j_{2}\in \mathbb{Z}} 
		h_{(-k_{1},j_{2})}\; 
		e^{i\lambda_{(-k_{1},k_{2})}T}
		\left( \int\limits_{0}^{T}
		e^{-i\lambda_{(-k_{1},k_{2})} t}\;
		\overline{q_{(-k_{1},j_{2})}(t)} 
		\;dt \right)
		\widehat{g_{2}(y)}(k_{2}-j_{2})\;
		m_{1}^{-k_{1},-k_{1}}\\
		&\quad + 
		\sum_{j_{1}\in \mathbb{Z}} 
		h_{(j_{1},k_{2})}\; 
		e^{i\lambda_{(-k_{1},k_{2})}T} 
		\left( \int\limits_{0}^{T}
		e^{-i\lambda_{(-k_{1},k_{2})} t}\;
		\overline{q_{(j_{1},k_{2})}(t)} 
		\;dt \right) 
		\widehat{g_{1}(x)}(-k_{1}-j_{1})\;
		m_{2}^{k_{2},k_{2}}.
	\end{split}
\end{equation*}
As before, from \eqref{dualbg}-\eqref{defq}, all terms in the sums are zero, except the ones for which  $j_{1}, j_{2}$ solves 
 $$\lambda_{(-k_{1},j_{2})}=\lambda_{(-k_{1},k_{2})},$$
$$\lambda_{(j_{1},k_{2})}=\lambda_{(-k_{1},k_{2})}.$$
In view of hypothesis $(H2)$ again, we must have $j_{2}=k_{2}$ and $j_{1}=\mp k_{1}$. Hence, 
\begin{equation*}
	\begin{split}
		2\pi \widehat{u_{1}}(-k_{1},k_{2})
		&=
		h_{(-k_{1},k_{2})}\; 
		e^{i\lambda_{(-k_{1},k_{2})}T}
		\widehat{g_{2}(y)}(0)\;
		\frac{1}{2 \pi}
		 +  
		h_{(-k_{1},k_{2})}\; 
		e^{i\lambda_{(-k_{1},k_{2})}T}  
		\widehat{g_{1}(x)}(0)\;
		\frac{1}{2 \pi}\\
&\quad + 
h_{(k_{1},k_{2})}\; 
e^{i\lambda_{(-k_{1},k_{2})}T} 
\widehat{g_{1}(x)}(-2k_{1})\;
\frac{1}{2 \pi},	
	\end{split}
\end{equation*}
or, which is the same,
\begin{equation}\label{h formg12}
	\begin{split}
		(2\pi)^{2} e^{-i\lambda_{(-k_{1},k_{2})}T} \widehat{u_{1}}(-k_{1},k_{2})
		&=
		h_{(-k_{1},k_{2})}\; 
		\left(
		\widehat{g_{2}(y)}(0)
		+   
		\widehat{g_{1}(x)}(0)\right)
		+
		h_{(k_{1},k_{2})}\;  
		\widehat{g_{1}(x)}(-2k_{1}).
	\end{split}
\end{equation}
It follows from \eqref{h formg9} and \eqref{h formg12} that we must solve the linear system
$$
\begin{cases}
	h_{(k_{1},k_{2})}\; 
	\left(
	\widehat{g_{2}(y)}(0)
	+   
	\widehat{g_{1}(x)}(0)\right)
	+
	h_{(-k_{1},k_{2})}\;  
	\widehat{g_{1}(x)}(2k_{1})
	=		(2\pi)^{2} e^{-i\lambda_{(k_{1},k_{2})}T} \widehat{u_{1}}(k_{1},k_{2}), & \\
h_{(-k_{1},k_{2})}\; 
\left(
\widehat{g_{2}(y)}(0)
+   
\widehat{g_{1}(x)}(0)\right)
+
h_{(k_{1},k_{2})}\;  
\widehat{g_{1}(x)}(-2k_{1})
=	(2\pi)^{2} e^{-i\lambda_{(-k_{1},k_{2})}T} \widehat{u_{1}}(-k_{1},k_{2}). & 	
\end{cases}
$$
To see that such a system has a (unique) solution, using \eqref{med1}-\eqref{med2}, we may write it as
\begin{equation*}
	\begin{pmatrix}
		\displaystyle{\frac{1}{\pi}} & \widehat{g_{1}(x)}(2k_{1})\\
		\widehat{g_{1}(x)}(-2k_{1}) & 	\displaystyle{\frac{1}{\pi}}
	\end{pmatrix}
\begin{pmatrix}
	h_{(k_{1},k_{2})} \\
	h_{(-k_{1},k_{2})}
\end{pmatrix}
= (2 \pi)^{2}
\begin{pmatrix}
	e^{-i\lambda_{(k_{1},k_{2})}T} \widehat{u_{1}}(k_{1},k_{2})\\
	e^{-i\lambda_{(-k_{1},k_{2})}T} \widehat{u_{1}}(-k_{1},k_{2})
\end{pmatrix}.
\end{equation*}
If we set 
$$M:=	\begin{pmatrix}
	\displaystyle{\frac{1}{\pi}} & \widehat{g_{1}(x)}(2k_{1})\\
	\widehat{g_{1}(x)}(-2k_{1}) & 	\displaystyle{\frac{1}{\pi}}
\end{pmatrix},$$
then 
$$\text{det}(M)=\displaystyle{\frac{1}{\pi^{2}}}-
\widehat{g_{1}(x)}(2k_{1})\;
\overline{\widehat{g_{1}(x)}(2k_{1})}= \displaystyle{\frac{1}{\pi^{2}}}-
\left|\widehat{g_{1}(x)}(2k_{1})\right|^{2}=:d_{k_{1}}.$$
Now, observe that since $g_1$ is a non-negative function, from \eqref{med1}, we deduce, 
$$
\left|\widehat{g_{1}(x)}(2k_{1})\right|\leq \frac{1}{2\pi} \int_0^{2\pi}g_1(x)dx=\frac{1}{2\pi}.
$$
Hence,
$$d_{k_{1}}\geq\frac{1}{\pi^{2}}-\frac{1}{4 \pi^{2}}=\frac{3}{4 \pi^{2}},\;\forall k_{1}\in \mathbb{Z},$$
and the matrix $M$ is invertible with 
\begin{equation}\label{invertmat}
	M^{-1}=
	\begin{pmatrix}
		\dfrac{1}{\pi d_{k_{1}}} & -\dfrac{\widehat{g_{1}(x)}(2k_{1})}{d_{k_{1}}}\\\\
		-\dfrac{\widehat{g_{1}(x)}(-2k_{1})}{d_{k_{1}}} &
		\dfrac{1}{\pi d_{k_{1}}}
	\end{pmatrix}.
\end{equation}
This implies that the above system has a solution and, in addition, there exists a constant $D$, independent of $k_{1}\in \mathbb{Z}^{\ast},$ such that
\begin{equation*}
	\|M^{-1}\|\leq D,
\end{equation*}
where $\|M^{-1}\|$ is the Euclidean norm of the matrix $M^{-1}.$\\

\noindent
{\bf{Step 3.}}
The function $h$ defined by
\eqref{thecontrol} with $h_{(0,0)}=0,$ $h_{(k_{1},0)}$ given  by \eqref{h formg3}, $h_{(0,k_{2})}$ given by \eqref{h formg6}, and $h_{(k_{1},k_{2})}$ given as the solution of
\begin{equation}\label{hforma}
	\begin{pmatrix}
		h_{(k_{1},k_{2})} \\
		h_{(-k_{1},k_{2})}
	\end{pmatrix}
	= M^{-1}
	\begin{pmatrix}
		(2 \pi)^{2} e^{-i\lambda_{(k_{1},k_{2})}T} \widehat{u_{1}}(k_{1},k_{2})\\
		(2 \pi)^{2} e^{-i\lambda_{(-k_{1},k_{2})}T} \widehat{u_{1}}(-k_{1},k_{2})
	\end{pmatrix},\;\;\text{for all}\;(k_{1},k_{2})\in \mathbb{Z}^{\ast}\times \mathbb{Z}^{\ast}.
\end{equation}
belongs to $L^{2}([0,T];H_{p}^{s}(\mathbb{T}^2))$.

Indeed, recall from Lemma \ref{basislemma} that $\{q_{ \mathbf{k} }\}_{ \mathbf{k} \in \mathbb{J}}$ is a Riesz basis for $H$. Thus, from {\cite[Theorem 7.13]{9}} and the definition of $q_{ \mathbf{k} }$, $\mathbf{k}\in \mathbb{Z}^2$ it follows that $\{q_{ \mathbf{k} }\}_{\mathbf{k}\in \mathbb{Z}^2}$ is a bounded sequence in $L^2([0,T])$. Hence, in view of  \eqref{thecontrol}, we deduce the existence of a positive constant $C$ such that 
\begin{align}\label{hes}
	\begin{split}
		\|h\|^{2}_{L^{2}([0,T];H_{p}^{s} (\mathbb{T}^{2}))}
		&=\sum_{\mathbf{k}\in \mathbb{Z}^{2}} (1+|\mathbf{k}|)^{2s}|h_{\mathbf{k}}|^2 \int_{0}^{T}|q_{\mathbf{k}}(t)|^{2}\;dt\\
		&\leq C \sum_{\mathbf{k}\in \mathbb{Z}^{2}}(1+|\mathbf{k}|)^{2s} |h_{\mathbf{k}}|^{2}\\
		&=C \sum_{k_{1}\in \mathbb{Z}^{\ast}} (1+|(k_{1},0)|)^{2s}
		|h_{(k_{1},0)}|^{2}\\
		&\quad+ C\sum_{k_{2}\in  \mathbb{Z}^{\ast}}(1+|(0,k_{2})|)^{2s} |h_{(0,k_{2})}|^{2}\\
		&\quad+ C\sum_{(k_{1},k_{2})\in \mathbb{Z}^{\ast}\times \mathbb{Z}^{\ast}} (1+|(k_{1},k_{2})|)^{2s}|h_{(k_{1},k_{2})}|^{2}.\\
	\end{split}
\end{align}
 From \eqref{hforma}, we infer that
$$
	|h_{(k_{1},k_{2})}|^{2}\leq \|M^{-1}\|^{2} (2\pi)^{4}\left(|\widehat{u_{1}}(k_{1},k_{2})|^{2}+|\widehat{u_{1}}(-k_{1},k_{2})|^{2}\right),\;\forall (k_{1},k_{2})\in \mathbb{Z}^{\ast}\times \mathbb{Z}^{\ast}.
$$
Therefore,
\begin{equation}\label{est}
	\begin{split}
(1+|(k_{1},k_{2})|)^{2s}	|h_{(k_{1},k_{2})}|^{2}&\leq D^{2} (2\pi)^{4}(1+|(k_{1},k_{2})|)^{2s} |\widehat{u_{1}}(k_{1},k_{2})|^{2}\\
&\quad+ D^{2} (2\pi)^{4}|(1+|(-k_{1},k_{2})|)^{2s} |\widehat{u_{1}}(-k_{1},k_{2})|^{2},
   \end{split}
\end{equation}
for all $(k_{1},k_{2})\in \mathbb{Z}^{\ast}\times \mathbb{Z}^{\ast}.$
Thus,  identities  \eqref{h formg3}, \eqref{h formg6} and \eqref{est} imply
\begin{equation}\label{hform5g}
	\begin{split}
		\|h\|^{2}_{L^{2}([0,T];H_{p}^{s} (\mathbb{T}^{2})) }
		&\leq C (2\pi)^{6} \sum_{k_{1}\in \mathbb{Z}^{\ast} } (1+|(k_{1},0)|)^{2s}
|\widehat{u_{1}}(k_{1},0)|^{2}\\
&\quad+ C (2\pi)^{6}\sum_{k_{2}\in  \mathbb{Z}^{\ast}} (1+|(0,k_{2})|)^{2s} |\widehat{u_{1}}(0,k_{2})|^{2}\\
&\quad+ 2 C (2\pi)^{4} D^{2}\sum_{(k_{1},k_{2})\in \mathbb{Z}^{\ast}\times \mathbb{Z}^{\ast}} (1+|(k_{1},k_{2})|)^{2s}| \widehat{u_{1}}(k_{1},k_{2})|^{2}\\	
&\leq \nu^{2} \sum_{\mathbf{k}\in \mathbb{Z}^{2}} (1+|\mathbf{k}|)^{2s}| \widehat{u_{1}}(\mathbf{k})|^{2},
	\end{split}
\end{equation}
where $\nu^{2}=3\max\{C(2\pi)^{6},2C(2\pi)^{4}D^{2}\}.$
Since $u_1\in H^s_p(\mathbb{T}^{2})$ the above series converges. In addition \eqref{esth} holds (recall $u_{0}=0$). This completes the proof of the theorem.
\end{proof}

The proof of Theorem \ref{crtI1} is very similar with minor modifications, so we omit the details.\\

\begin{cor}\label{ope}
Assume  $s\geq 0$. Under the assumptions of Theorem \ref{crtI} or Theorem \ref{crtI1}, for any $T>\frac{2\pi }{\gamma'}$, there exists a {unique}  bounded linear operator
$$ 
	\begin{cases}
		\Phi: H_p^{s}(\mathbb{T}^{2})\times H_p^{s}(\mathbb{T}^{2}) \longrightarrow L^{2}([0,T];H_p^{s}(\mathbb{T}^{2}))& \\
		(u_{0},u_{1})\longmapsto \Phi (u_{0},u_{1})=:h
	\end{cases}
$$
such that
$$u_{1}=U(T)u_{0}+\int_{0}^{T} U(T-s)(G(\Phi(u_{0},u_{1})))(\cdot,\cdot,s) ds$$
and 
$$\|\Phi(u_{0},u_{1})\|_{L^{2}([0,T];H_p^{s}(\mathbb{T}^{2}))} \leq \nu \left( \|u_{0}\|_{H_p^{s}(\mathbb{T}^{2})}+\|u_{1}\|_{H_p^{s}(\mathbb{T}^{2})}\right),$$
for some positive constant $\nu.$
\end{cor}

\begin{rem}
	The constant $\nu$ in Corollary \ref{ope} depends only on $s,\;g_{1}$ and $T$ (resp. $s,\;g_{2}$ and $T$) under assumptions of 
	 Theorem \ref{crtI} (resp. Theorem \ref{crtI1}).
\end{rem}

\begin{cor}\label{Observabi}
	Let	$s\geq 0.$ Under assumptions of Theorem \ref{crtI} or Theorem \ref{crtI1}, for any $T>\frac{2\pi }{\gamma'}$   there exists $\delta>0$ such that
	\begin{equation}\label{obserineq}
		\int_{0}^{T} \left\|G^{\ast} U(-t)^{\ast} \phi \right\|^{2}_{H_p^{s}(\mathbb{T}^{2})}(t) dt \geq \delta^{2} \|\phi\|^{2}_{H_p^{s}(\mathbb{T}^{2})},\;\forall \phi \in H_p^{s}(\mathbb{T}^{2}),
	\end{equation}
	where the constant $\delta$ depends only on $s,g_{1},$ and $T$ (resp. $s,g_{2},$ and $T$) under assumptions of 
	Theorem \ref{crtI} (resp. Theorem \ref{crtI1}).
\end{cor}
\begin{proof}
This is a consequence  of the	Hilbert Uniqueness Method (HUM) due to J.-L. Lions \cite{Lions}. Actually, as is well known, the exact controllability is equivalent to the observability inequality \eqref{obserineq}. See for instance \cite[Theorem 2.3]{Liu} or \cite[Theorem 2.4]{Lionel Rosier}.
\end{proof}

\begin{rem}
	If $\gamma'= +\infty,$ then Corollaries \ref{ope} and \ref{Observabi} are valid for any time $T>0.$
\end{rem}

\section{Stabilization Results}\label{stabsec}
In this section we prove the exponential stabilization result stated in Theorem  \ref{estabilization}. First, we show if $K$ is a bounded operator in $H_p^{s}(\mathbb{T}^{2})$ then system  \eqref{stabp} is globally well-posed in $H_p^{s}(\mathbb{T}^{2}),$ $s\geq 0.$

\begin{thm}\label{well-posedness}
	Let $u_{0}\in H_p^{s}(\mathbb{T}^{2}),$ $s\geq 0.$ Then the IVP \eqref{stabp} has a unique (mild) solution 
	$$u\in C([0,\infty);H_p^{s}(\mathbb{T}^{2})).$$
\end{thm}
\begin{proof}
	We know, from Theorem \ref{wellpose}, that operator  $\partial_{x} \mathcal{L}$, with domain $H^{s+r}_{p}(\mathbb{T}^{2})$ is the infinitesimal generator of a unitary group in $H^{s}_{p}(\mathbb{T}^{2})$. Hence, it also generates of a $C_{0}$-semigroup $\left\{ U(t)\right\}_{t\geq 0}$. We also know that $GK$ is a bounded linear operator on $H^{s}_{p}(\mathbb{T}^{2}).$ From the semigroup theory (see \cite[page 76]{Pazy}), we get that operator $\partial_{x} \mathcal{L}+GK,$ which is a perturbation of $\partial_{x} \mathcal{L}$ by a bounded linear operator, is the infinitesimal generator of a $C_{0}$-semigroup, say, 
	$\left\{ T(t)\right\}_{t\geq 0}$ on $H^{s}_{p}(\mathbb{T}^{2}).$ Consequently, \eqref{stabp} has a unique mild solution.
\end{proof}

\begin{proof}[Proof of Theorem \ref{estabilization}]
 The well-posedness of IVP \eqref{estag} is given by Theorem \ref{well-posedness}. Then, Theorem \ref{estabilization} is a direct consequence of Corollary \ref{Observabi} and the classical principle: Exact controllability implies exponential stabilizability for conservative control systems (see \cite[Theorem 2.3-2.4]{Liu} and \cite[Theorem 2.1]{Slemrod}). Actually, according to \cite{Slemrod} one can choose
	$$K_{\lambda}=-G^{\ast}D_{T,\lambda}^{-1},$$
	where, for some $T>\frac{2\pi}{\gamma'},$
	\begin{equation}\label{feedbaclaw}
		D_{T,\lambda}\phi= \int_{0}^{T}e^{-2\lambda \tau} U(-\tau)GG^{\ast}U(-\tau)^{\ast} \phi\; d\tau,\;\;\;\forall \phi \in H_{p}^{s}(\mathbb{T}^{2}),
	\end{equation}
	and $U(t)$ is the $C_{0}$-semigroup generated by $\partial_{x}\mathcal{L}.$
\end{proof}

\section{Applications}\label{appsec}
In many situation, \textit{internal waves} arise due to the gravitational effects, at the interface of two layers in a stratified fluid. Several theoretical models exist which govern the evolution of long internal waves with small amplitudes in such cases. When the height of the heavier fluid is much larger than that of the upper layer, the motion is described by the Benjamin-Ono equation (BO) \cite{T Benjamin, H Ono}:
\begin{equation}\label{BO}
	\partial_{t}u-\mathcal{H}\partial^{2}_{x}u+u \partial_{x}u=0,\;\;x\in \mathbb{R},\;t>0,
\end{equation}
where $\mathcal{H}$ denotes the Hilbert transform.
Equation \eqref{BO} may also be viewed as a general model for the propagation of  weakly nonlinear long waves incorporating the lowest-order effects of nonlinearity and non-local dispersion and it turns out  to be important in many others physical situations (see, for instance, \cite{Danov and Ruderman, Ishimori, Matsuno and Kaup}).

On the other hand, when the total depth of the a fluid is very small, the motion is governed by the Korteweg-de Vries (KdV) equation
\begin{equation}\label{KdV}
	\partial_{t}u+\partial^{3}_{x}u+u \partial_{x}u=0,\;\;x\in \mathbb{R},\;t>0,
\end{equation}
derived in \cite{Korteweg and Vries} as a model for the propagation of long one dimensional surface gravity waves with small amplitude in a shallow channel of water. The KdV equation has a very rich structure from the mathematical point of view and it has also been derived in several other physical context (see, for instance, \cite{Ablowitz and Clarkson}).

In both situations above, when transversal effects must also be considered, the resulting equations are bidimensional. Hence, in this section, we present some particular examples of bidimensional  dispersive PDE's, where the general control  theory developed in this work can be applied to their linear counterpart.

\subsection{The Zakharov-Kuznetsov (ZK) equation:}One of the most accepted generalization of the KdV equation in two dimensions is the Zakharov-Kuznetsov (ZK) equation:
\begin{equation}\label{ZK}
	\partial_{t}u+\partial_{x}\Delta u+u \partial_{x}u=0,\;\;(x,y)\in \mathbb{R}^{2},\;t>0,
\end{equation}
where $\Delta$ denotes the bidimensional Laplacian, that is, $\Delta=\partial_{x}^{2}+\partial_{y}^{2}.$ 
Equation \eqref{ZK} models ion-acoustic waves propagating in a low-pressure magnetized plasma. It was derived in \cite{Zakharov and Kuznetsov} where the existence and stability for circularly symmetric soliton solutions were established. Questions of local well-posedness for \eqref{ZK} in the Sobolev spaces $H^s(\mathbb{R}^2)$ may be found, for instance in, \cite{f}, \cite{kinoshita}, \cite{LP}, \cite{mp}. The initial-value problem posed on the two dimensional torus was studied in  \cite{lipa}. In addition,  in \cite{MC and LR}  the authors addressed the exact controllability of the linear ZK equation on a rectangle with a left Dirichlet boundary control  by using the flatness approach.

Here we address the exact controllability associated with the linear equation
 \begin{equation}\label{LZK}
	\begin{cases}
		\partial_{t}u+\partial_{x}\Delta u=Gh, \;\;\;(x,y)\in \mathbb{T}^{2},\;t>0,\\
		u(x,y,0)=u_0(x,y).
	\end{cases}	
\end{equation}
In order to set \eqref{LZK} as in \eqref{FEQ1}, we define $\mathcal{L}=-\Delta$ so that 
$$b(\mathbf{k}):=|\mathbf{k}|^{2}=k_{1}^{2}+k_{2}^{2}, \;\;\text{for all}\;\mathbf{k}= (k_{1},k_{2}) \in \mathbb{Z}^{2},
$$
and the eigenvalues of $\partial_{x}\mathcal{L}$ are $i\lambda_{\mathbf{k}}$ with
\begin{equation*}
	\lambda_{\mathbf{k}}:=k_{1}b(\mathbf{k})=k_{1}(k_{1}^{2}+k_{2}^{2}).
\end{equation*}
Clearly,
$$|b(\mathbf{k})|\leq |\mathbf{k}|^{2},\;\;\text{for all}\;\mathbf{k}=(k_{1},k_{2}) \in \mathbb{Z}^{2}$$
and \eqref{Scont} holds.

Also, it is easy to check that $(H3)$ holds and the value of $\gamma$ in \eqref{gammacon1} is equal to $1.$
Note that, if $n \in \mathbb{N},$  then  $\lambda_{(2^{2n},0)}\to \infty,$ $\lambda_{(1,2^{3n})}\to \infty,$ as $n \to \infty$ and  $\lambda_{(2^{2n},0)} \neq \lambda_{(1,2^{3n})}$ with
$$\left|\lambda_{(2^{2n},0)}- \lambda_{(1,2^{3n})}\right|
=|2^{6n}-1-2^{6n}|= 1.$$
Hence, $\gamma'$ defined in \eqref{gammalinhacond1} is also equal to $1.$ 

Applying Theorem \ref{crtI1}
we conclude that system \eqref{LZK} is exactly controllable in any time $T>2\pi$ in the Sobolev space $H_{p}^{s}(\mathbb{T}^{2})$ with $s\geq 0,$ where the control function $h$ is given by \eqref{thecontrol}. Also, Theorem \ref{estabilization} holds and the system  \eqref{LZK} is exponentially stabilizable with any decay rate $\lambda>0.$

\subsection{The 2D Benjamin-Ono (2D-BO) equation:}\label{BOexample}
In this subsection, we consider a two-dimensional extension of the BO equation, which reads as
\begin{equation}\label{2D-BO}
	\partial_{t}u-\mathcal{H}^{(x)}\partial^{2}_{xy} u+u \partial_{y}u=0,\;\;(x,y)\in \mathbb{R}^{2},\;t>0,
\end{equation}
where $\mathcal{H}^{(x)}$ denotes the Hilbert transform with respect to the $x$-variable, that is, via Fourier transform,
$$
\widehat{\mathcal{H}^{(x)}u}(\xi,\eta)=-i\;\text{sng}(\xi)\widehat{u}(\xi,\eta), \qquad (\xi,\eta)\in\mathbb{R}^2.
$$
From the mathematical point of view, local and global well-posedness for \eqref{2D-BO} have been studied in \cite{Aniura} and \cite{Aniura Tesis}.

The control equation associated to the linear part of \eqref{2D-BO} on the periodic setting reads as follows:
\begin{equation}\label{2DBO}
	\partial_{t}u-\mathcal{H}^{(x)}\partial^{2}_{xy}u=Gh,\;\; u(x,y,0)=u_{0}(x,y), \; \;\;(x,y)\in \mathbb{T}^{2},\;t>0.
\end{equation}
In this case the operator $\mathcal{L}$ takes the form $\mathcal{L}=\mathcal{H}^{(x)}\partial_{y},$ where the Hilbert transform $\mathcal{H}^{(x)}$ in the frequency space is given  by
$$\widehat{ \mathcal{H}^{(x)}u }(k_{1},k_{2}):=-i\; \text{sng}(k_{1}) \widehat{u}(k_{1},k_{2}),\;\;\mathbf{k}=(k_{1},k_{2})\in \mathbb{Z}^2.$$
Therefore, $$b(\mathbf{k})=k_{2}\;\text{sgn}(k_{1}),$$
and the eigenvalues of operator $\partial_{x}\mathcal{L}$ have  the form $i\lambda_{\mathbf{k}}$ with 
\begin{equation}\label{2DBO1}
	\lambda_{\mathbf{k}}:=k_{1}b(\mathbf{k})=|k_{1}| k_{2},\;\;\;\mathbf{k}\in \mathbb{Z}^{2}.
\end{equation}

In what follows we shall show that Theorems \ref{crtI}, and \ref{estabilization} can be applied to prove that  \eqref{2DBO} is exactly controllable in any time $T>2\pi,$ and exponentially stabilizable with any given decay rate in the Sobolev space $H_{p}^{s}(\mathbb{T}^{2})$, $s\geq 0.$ Indeed, first of all note that
$$
|b(\mathbf{k})|\leq |\mathbf{k}|,\;\;
\;\mathbf{k} \in \mathbb{Z}^2,
$$
and \eqref{Scont} is true with $r=2$. From 
\eqref{2DBO1} it is clear that $(H2)$ holds. Additionally, for any $\mathbf{k},\mathbf{k}' \in \mathbb{J}, $  
$$\left|\lambda_{\mathbf{k}}- \lambda_{\mathbf{k}'}\right|
=\Big| |k_{1}| k_{2}-|k'_{1}|  k'_{2} \Big|\geq 1.$$
Also, note that, if  $k_{1}\to \infty$ then $\lambda_{(k_{1}+1,1)}\to \infty,$ $\lambda_{(k_{1},1)}\to \infty,$ and 
$\lambda_{(k_{1}+1,1)} \neq \lambda_{(k_{1},1)}$ with
$$\left|\lambda_{(k_{1}+1,1)}- \lambda_{(k_{1},1)}\right|
=\left||k_{1}+1| -|k_{1}|\right|= 1.$$
Therefore, $\gamma$ and $\gamma'$ defined respectively by \eqref{gammacon} and \eqref{gammalinhacond} are, in this case, equal to 1. The result follows as desired.

\subsection{The Benjamin-Ono-Zakharov-Kuznetsov (BOZK) equation:} 
Another model that may be seen as a two-dimensional extension of the BO equation is the so called BOZK equation:
\begin{equation}\label{BLZKnl}
	\partial_{t}u-\mathcal{H}^{(x)}\partial_{x}^{2}u +\partial_{x} \partial_{y}^{2} u+u\partial_{x}u=0, \; \;\;(x,y)\in \mathbb{R}^{2},\;t>0.
\end{equation}
The equation in \eqref{BLZKnl} was introduced in \cite{lmsv} \cite{jcms}, and it has
applications to electromigration in thin nanoconductors on a dielectric substrate. Local and global well-posedness for the Cauchy problem associated with \eqref{BLZKnl} in Sobolev spaces was studied, for instance, in \cite{cunha}, \cite{cunha1}, and \cite{rib}.

In this subsection we investigate the control and stabilization properties of linear BOZK equation:
\begin{equation}\label{BOZK}
	\partial_{t}u-\mathcal{H}^{(x)}\partial_{x}^{2}u +\partial_{x} \partial_{y}^{2} u=Gh,\;\; u(x,y,0)=u_{0}(x,y), \; \;\;(x,y)\in \mathbb{T}^{2},\;t>0.	
\end{equation}
Here, we consider the operator $\mathcal{L}$ defined in \eqref{FEQ1} as $\mathcal{L}:= \mathcal{H}^{(x)}\partial_{x} - \partial_{y}^{2}.$ Therefore, $$b(\mathbf{k})=|k_{1}|+k_{2}^{2}, \;\;\;\mathbf{k}=(k_{1},k_{2})\in \mathbb{Z}^{2},$$ and 
$$|b(\mathbf{k})|=|k_{1}|+k_{2}^{2}\leq |\mathbf{k}|+|\mathbf{k}|^{2} \leq 2 |\mathbf{k}|^{2},\;\;\text{for all}\; \mathbf{k}\in \mathbb{Z}^{2}. $$
The eigenvalues of operator $\partial_{x}\mathcal{L}$ are $i\lambda_{\mathbf{k}}$ with 
\begin{equation*}
	\lambda_{\mathbf{k}}:=k_{1}b(\mathbf{k})=k_{1}(|k_{1}|+k_{2}^{2}).
\end{equation*}

Next, we shall verify that (H3) holds. We easily check that $\lambda_{(k_{1},-k_{2})}=\lambda_{(k_{1},k_{2})}$ and  $\lambda_{(-k_{1},k_{2})}=-\lambda_{(k_{1},k_{2})}$ for all $(k_{1},k_{2}) \in \mathbb{Z}^{2}$. Also, for $(k_{1},k_{2}) \in \mathbb{Z}^{2}$,  with $k_{1}\neq 0$, it is clear that the unique entire solutions of the $j_{2}$-equation
$$\lambda_{(k_{1},j_{2})}=\lambda_{(k_{1},k_{2})}\iff k_{1}j_{2}^{2}=k_{1}k_{2}^{2},$$
are $j_{2}=\pm k_{2}.$ On the other hand, for $(k_{1},k_{2}) \in \mathbb{Z}^{2}$, with $k_{2}\neq 0,$ we analyze 
the entire solutions of the $j_{1}$-equation
$$\lambda_{(j_{1},k_{2})}=\lambda_{(k_{1},k_{2})} \iff j_{1}(|j_{1}|+k_{2}^{2})=k_{1}(|k_{1}|+k_{2}^{2}),$$
which can be rewritten in the following form
\begin{equation}\label{BOZK1}
 j_{1}|j_{1}|-k_{1}|k_{1}|+k_{2}^{2}(j_{1}-k_{1})=0.
\end{equation}
It is enough to assume $k_{1}\neq 0,$ because the unique solution of  \eqref{BOZK1} with $k_{1}=0$ is clearly $j_{1}=0$ and the desired result follows. Immediately, we observe that $j_{1}$ and $k_{1}$ should share the same sign. To see this, it suffices to note that  the  expression on left-hand side in \eqref{BOZK1} is strictly positive  if $j_{1}\geq 0$ and $k_{1}<0$  and  strictly negative when $j_{1}\leq 0$ and $k_{1}>0.$  Therefore, to solve equation \eqref{BOZK1} with $k_{1}>0,$ we may assume $j_{1}\geq 0$ to see that it is equivalent to
$$ 
(j_{1}-k_{1})(j_{1}+k_{1}+k_{2}^{2})=0,
$$
from which we obtain that the unique entire solution is $j_{1}=k_{1}.$ Similarly, when $k_{1}<0,$ equation \eqref{BOZK1} is equivalent to
$$ 
(j_{1}-k_{1})(-j_{1}-k_{1}+k_{2}^{2})=0,$$
and again the unique entire solution is $j_{1}=k_{1}.$
Consequently, $(H3)$ holds.

Finally, we note that $\gamma$ given by \eqref{gammacon1} is equal to 1. In addition, by taking $\lambda_{(k_{1},0)}$ and $\lambda_{(1,k_{1})}$ for any $0<k_{1} \in \mathbb{Z}$ we easily verify that $\lambda_{(k_{1},0)}\to \infty,$ $\lambda_{(1,k_{1})} \to \infty$ as $k_{1}\to \infty$ and $\lambda_{(k_{1},0)} \neq \lambda_{(1,k_{1})}$ with
$$|\lambda_{(k_{1},0)}-\lambda_{(1,k_{1})}|=1.$$
from which we infer $\gamma'=1$ (see \eqref{gammalinhacond1}).
Thus, we can apply Theorems \ref{crtI1} and \ref{estabilization}  to deduce that system \eqref{BOZK} is exactly controllable in any $T>2 \pi,$ and exponentially stabilizable with any decay rate in the Sobolev space $H_{p}^{s}(\mathbb{T}^{2})$, $s\geq 0.$

\subsection{The dispersion generalized Benjamin-Ono-Zakharov-Kuznetsov (dgBOZK) equation:} To finish our applications, we shall consider the dgBOZK equation
\begin{equation}\label{DGBOZK11}
	\partial_{t}u-D_{x}^{\alpha}\partial_{x}u +\partial_{x} \partial_{y}^{2} u+u\partial_xu=0,\;\;\;\;(x,y)\in \mathbb{R}^{2},\;t>0,	
\end{equation}
where $\alpha>0$ and $D_{x}^{\alpha}$ is defined via Fourier transform  as $\widehat{D_{x}^{\alpha}u}(\xi,\eta)=|\xi|^{\alpha}\widehat{u}(\xi,\eta).$
In the case $\alpha\in(1,2)$, equation \eqref{DGBOZK11} may be seen as an interpolation between the ZK and BOZK equations in the sense that in the limiting cases $\alpha=2$ and $\alpha=1$,  \eqref{DGBOZK11} reduces to ZK and BOZK equations, respectively. The interested reader will find some local and global well-posedness results for the associated Cauchy problem in  \cite{cunha2} and \cite{rib}.

As in the earlier examples, here we  study the control problem for the linear dgBOZK equation: 
\begin{equation}\label{DGBOZK}
	\partial_{t}u-D_{x}^{\alpha}\partial_{x}u +\partial_{x} \partial_{y}^{2} u=Gh,\;\; u(x,y,0)=u_{0}(x,y),\;\;(x,y)\in \mathbb{T}^{2},\;t>0,	
\end{equation}
where $\alpha>0$ and $D_{x}^{\alpha}$ is now defined  as $\widehat{D_{x}^{\alpha}u}(\mathbf{k})=|k_{1}|^{\alpha}\widehat{u}(\mathbf{k}).$ Thus,  the operator $\mathcal{L}$ reads as $\mathcal{L}:=D_{x}^{\alpha}-\partial_{y}^{2},$ so that 
\begin{equation*}
	b(\mathbf{k})=|k_{1}|^{\alpha}+k_{2}^{2},
\end{equation*}
and 
$$|b(\mathbf{k})|\leq |k_{1}|^{\alpha}+k_{2}^{2}\leq |\mathbf{k}|^{\alpha}+|\mathbf{k}|^{2}\leq \begin{cases}
	2|\mathbf{k}|^{2}, & \text{if} \;0<\alpha<2,\\
	2|\mathbf{k}|^{\alpha},& \text{if} \;\alpha \geq 2,
\end{cases}$$
 which means that \eqref{Scont} holds.
The eigenvalues of $\partial_{x}\mathcal{L}$ are $i\lambda_{\mathbf{k}}$ with
\begin{equation*}
	\lambda_{\mathbf{k}}:=k_{1}b(\mathbf{k})= k_{1}(|k_{1}|^{\alpha} +k_{2}^{2}).
\end{equation*}

Let us check that (H3) also holds here. Indeed, clearly the eigenvalues are even in the second variable and odd in the first one. Also, for any $\mathbf{k}=(k_{1},k_{2}) \in \mathbb{Z}^{2}$ given with $k_{1}\neq 0,$ it is easy to show that the unique entire solutions of the $j_{2}$-equation
$$\lambda_{(k_{1},j_{2})}=\lambda_{(k_{1},k_{2})}\iff k_{1}(|k_{1}|^{\alpha}+j_{2}^{2})= k_{1}(|k_{1}|^{\alpha}+k_{2}^{2}),$$
are $j_{2}=\pm k_{2}.$ 
On the other hand,
if $\mathbf{k}=(k_{1},k_{2}) \in \mathbb{Z}^{2}$ is such that $k_{2}\neq 0,$ we now analyze 
the entire solutions of the $j_{1}$-equation
\[
	\begin{split}
		\lambda_{(j_{1},k_{2})}=\lambda_{(k_{1},k_{2})} &\iff j_{1}(|j_{1}|^{\alpha}+k_{2}^{2})=k_{1}(|k_{1}|^{\alpha}+k_{2}^{2}),
\end{split}
\]
which can be rewritten as
\begin{equation}\label{dgBOZK3}
	j_{1}|j_{1}|^{\alpha}-k_{1}|k_{1}|^{\alpha}+k_{2}^{2}(j_{1}-k_{1})=0.
\end{equation}
Similar to the analysis for the BOZK equation, we may assume $k_{1}\neq 0$ and observe that $j_{1},$ $k_{1}$ share the same sign. Without loss of generality, let us assume $k_1>0$ (the case $k_1<0$ being similar). Therefore,  we may assume $j_{1}\geq 0$ and  \eqref{dgBOZK3} is equivalent to
\begin{equation}\label{key1}
	\begin{split}
		j_{1}^{\alpha+1}-k_{1}^{\alpha+1}+k_{2}^{2}(j_{1}-k_{1})=0.
	\end{split}
\end{equation}
Recall we want to show that \eqref{key1} has no other solution than $j_1=k_1$. Assume by contradiction the existence of another solution, say, with $j_1>k_1$. Then, from the Mean Value Theorem, for some $\theta$ between $k_{1}$ and $j_{1}$, we have
\[
		(\alpha+1) \theta^{\alpha}(j_{1}-k_{1})+k_{2}^{2}(j_{1}-k_{1})=0
\]
or
\[
		(j_{1}-k_{1}) [(\alpha +1) \theta^{\alpha}+k_{2}^{2}]=0.
\]
Since the expression between brackets is positive, this last identity is clearly a contradiction. This shows hypothesis $(H3)$ holds. 

Hence, Theorems \ref{crtI1} and \ref{estabilization} also apply in this case and
we conclude that system \eqref{DGBOZK} is exactly controllable at any time $T>2\pi$ and exponentially stabilizable with any decay rate.


\subsection*{Acknowledgment}
This paper was written when the first author had a postdoctoral position
at IMECC-UNICAMP, whose hospitality he gratefully acknowledges.
He also acknowledges the financial support from FAPESP/Brazil grant 2020/14226-4.
The second author is partially supported by CNPq/Brazil grant 303762/2019-5 and FAPESP/Brazil grant
2019/02512-5.


\end{document}